\documentclass[11pt]{article}
\usepackage{amsmath,amsthm, amssymb, amsfonts,accents,nccmath,mathrsfs}
\usepackage{enumitem}
\usepackage{array}
\usepackage[toc,page]{appendix}
\usepackage[english]{babel}
\usepackage{dsfont}
\usepackage{booktabs}
\usepackage[linesnumbered,commentsnumbered,ruled,vlined]{algorithm2e}
\usepackage{tikz}
\usepackage{tikz-network}
\usepackage[utf8]{inputenc}
\usepackage[english]{babel}
\usepackage{mathrsfs}

\usepackage{caption}
\usepackage{subcaption}

\usepackage[switch]{lineno}
\usepackage{blindtext}
\usepackage{lipsum}

\usepackage{hyperref}
\hypersetup{
    colorlinks=true,
    linkcolor=blue,
    filecolor=darkmagenta,      
    urlcolor=cyan,
    citecolor=red
} 
\makeatletter
\tikzset{
    dot diameter/.store in=\dot@diameter,
    dot diameter=2pt,
    dot spacing/.store in=\dot@spacing,
    dot spacing=13pt,
    dots/.style={
        line width=\dot@diameter,
        line cap=round,
        dash pattern=on 0pt off \dot@spacing
    }
}
\makeatother

\makeatletter
\def\BState{\State\hskip-\ALG@thistlm}
\makeatother
\numberwithin{equation}{section}

\newtheorem{theorem}{Theorem}[section]
\newtheorem{cor}[theorem]{Corollary}
\newtheorem{lem}[theorem]{Lemma}
\newtheorem{prop}[theorem]{Proposition}
\newtheorem{defn}[theorem]{Definition}
\newtheorem{rem}[theorem]{Remark}
\newtheorem{ex}[theorem]{Example}

\newtheorem{result}[theorem]{Result}

\newcommand{\R}{{\mathbb R}}

\newcommand{\ds}{\displaystyle}

\title{ Trees with Matrix Weights:\\Laplacian Matrix and  Characteristic-like Vertices}
\author{Swetha Ganesh\footnote{Department of Computer Science And Automation, IISc Bangalore, Bangalore, Karnataka-560012, India. \newline \indent Email: swethaganesh@iisc.ac.in}  \quad and \quad Sumit Mohanty\footnote{Humanities and Applied Sciences, IIM Ranchi,  Suchana Bhawan, Audrey House Campus, Meur's Road, Ranchi, Jharkhand-834008, India. \indent   Email:  sumitmath@gmail.com, sumit.mohanty@iimranchi.ac.in}}

\date{}

\begin{document}

\maketitle

\begin{abstract}
It is known that there is an alternative characterization of characteristic vertices for trees with positive weights on their edges via  Perron values and Perron branches. Moreover, the algebraic connectivity of a tree with positive edge weights can be expressed in terms of Perron value.  

In this article, we consider trees with matrix weights on their edges. More precisely, we are interested in trees with the following classes of matrix edge weights:
\begin{enumerate}
\item[$1.$] positive definite matrix weights, 

\item[$2.$] lower (or upper) triangular matrix weights with positive diagonal entries.
\end{enumerate}
For trees with the above classes of matrix edge weights, we define   Perron values and Perron branches. Further, we have shown the existence of vertices satisfying properties analogous to the properties of characteristic vertices of trees with positive edge weights in terms of Perron values and Perron branches, and we call such vertices characteristic-like vertices. In this case,  the eigenvalues of the Laplacian matrix are nonnegative, and we obtain a lower bound for the first non-zero eigenvalue of the Laplacian matrix in terms of Perron value. Furthermore, we also compute the Moore-Penrose inverse of the Laplacian matrix of a tree with nonsingular matrix weights on its edges. 
\end{abstract}



\section{Introduction and  Motivation  }\label{sec:intro}
Let $G=(V, E)$ be a simple graph, with $V$ as the set of vertices and $E$ as the set of edges in $G$. For $u,v \in V,$  we write $u\sim v$ if $u$ and $v$ are adjacent in $G$, and $u\nsim v$  otherwise.  We write, $\deg(v)$ to denote the degree of  the vertex $v$ and  $\mathcal{P}(u,v)$ to denote the  path joining vertices   $u$  and $v$.

Given a graph $G=(V, E)$ on $n$ vertices,  if each edge $e \in E$ is associated with a positive number $W(e)$, called the weight of $e$, then the Laplacian matrix $L(G)=[l_{uv}]$ is an $n\times n$ matrix (we simply write $L$ if there is no scope for confusion), and is defined as follows:  for $u,v \in V$, if $u\neq v$,  then $l_{uv}$  is $0$ if $u\nsim v$, and $l_{uv}$  is $-W(e)$ if $u\sim v$ and $e$ is the edge between them; finally if $u=v$, $l_{vv}$ is the sum of the weights of the edges in $G$ which are incident with the vertex $v$. It is well known that $L(G)$ is a symmetric  positive semidefinite matrix. The column vector with  constant value for each of its entries (constant vector)   is an eigenvector of $L(G)$ corresponding to the smallest eigenvalue $0$. In~\cite{Fiedler1}, Fiedler proved that  the second smallest eigenvalue of $L(G)$, say $\mu(G)$, is positive if and only if $G$ is connected. Since $\mu(G)$  provides an algebraic measure of the connectivity of $G$, it is named as algebraic connectivity of $G$. An eigenvector $\mathbf{y}$ of $L(G)$, corresponding  to the algebraic connectivity $\mu(G)$ is called Fiedler vector. Further, for any vertex $v\in V$,  we write $\mathbf{y}_v$ to denote the $v^{th}$ entry of  $\mathbf{y}$.

In particular, for a given tree $T$ with positive weights on its edges, there is an interesting result that gives some insight into the structure of the eigenvectors corresponding to the algebraic connectivity of $T$. This result was first proved for trees, where all the edge weights are equal to $1$ in~\cite{Fiedler2}. However, it is also valid for trees with positive weights.
 
\begin{prop}~\cite{Fiedler2}
Let $T=(V, E)$ be a  tree with positive weights on its edges. Let $L$ be the  Laplacian matrix of $T$ with algebraic connectivity $\mu(T)$ and  $\mathbf{y}$ be an eigenvector of $L$ corresponding to the algebraic connectivity $\mu(T)$. Then, exactly one of the following cases occurs:
\begin{enumerate}
    \item[(a)] No entry of $\mathbf{y}$ is $0.$ In this case, there is a unique pair of vertices $u$ and $v$ such that $u$ and $v$ are adjacent in $T$, with  $\mathbf{y}_u>0$ and $\mathbf{y}_v<0$. Further, the entries of $\mathbf{y}$ are increasing along any path in $T$ which starts at $u$ and does not contain $v$, while the entries of $\mathbf{y}$ are decreasing along any path in $T$ which starts at $v$ and doesn't contain $u$.
    \item[(b)] Some entry of  $\mathbf{y}$ is $0.$ In this case, the subgraph of $\, T$ induced by the set of vertices corresponding to 0's in $\mathbf{y}$ is connected. Moreover, there is a unique vertex $x$ such that $\mathbf{y}_x=0,$ and $x$ is adjacent to a vertex $w$ with $\mathbf{y}_w\neq 0.$ The entries of $\mathbf{y}$ are either increasing, decreasing, or identically $0$ along any path in $T$ which starts at $x.$
\end{enumerate}
\end{prop}
 
A  tree with positive weights on its edges is said to be of Type I if $(b)$ holds, and Type II if $(a)$ holds. If $T$ is of Type I, Fiedler defines the characteristic vertex as the special vertex $x$ referred to in $(b)$, whereas if $T$ is of Type II, he shows that $T$ has two characteristic vertices, namely the special vertices $u$ and $v$ referred to in $(a)$, and we call the edge between the vertices $u$ and $v$  is the characteristic edge of $T$. In~\cite{Merris},  it was shown that the characteristic vertex (or vertices) of $T$ is (are) independent of the choice of the eigenvector $\mathbf{y}$ corresponding to the algebraic connectivity $\mu(T)$. The above understanding of the characteristic vertices of trees and their relations with the algebraic connectivity has created a great deal of interest amongst researchers, and many interesting results have been obtained for trees (for example, see~\cite{Grone1,Grone2,Merris}).

Let $T$ be a tree with positive weights on its edges and let $\mathcal{C}_T$  denote the set of characteristic vertices of $T$. Then, $|\mathcal{C}_T|=1$ or $2$,  depending on whether $T$ contains a characteristic vertex or a characteristic edge, respectively.

Before proceeding further, we first introduce a few notations and then recall a few results from matrix theory which will be used time and again in this article.  Let $ \mathds{1}$, $I$, and $J$  denote the column vector of all ones, the identity matrix, and the matrix of all ones, respectively. We write $\mathbf{0}_{m \times n}$ to represent the zero matrix of order $m \times n$ and simply write $\mathbf{0}$ if there is no scope for confusion with the order of the matrix. Given a matrix $A$, we use $A^T$, Range($A$) and Null($A$) to denote the transpose, range and null space of the matrix $A$, respectively.  If $A$ is a square matrix, then the set of eigenvalues of $A$ is called the spectrum of $A$,  denoted by $\sigma(A)$ and the spectral radius of $A$, denoted by $\rho(A)$ is defined as  $\ds \rho(A)= \max_{\lambda \in \sigma(A) } |\lambda|.$ By Perron-Frobenius theory, if $A$ is an entrywise positive square matrix, then  $\rho(A)$ is the largest eigenvalue of $A$. Moreover, $\rho(A)$ is a simple eigenvalue of $A$ and is called the Perron value of $A$. Next, we state a result that compares the spectral radius of two nonnegative matrices, which is an application of Perron-Frobenius theory (for details see~\cite{Minc}).

\begin{theorem}~\cite[Corollary $2.2$]{Minc}\label{thm:sp-comp}
Let $A$ be an entrywise positive square matrix and $B$ be a nonnegative matrix  of the same order as $A$. If $A-B$ is a nonnegative matrix with at least one positive entry, then $\rho(A)> \rho(B)$.
\end{theorem}
Finally, given a real symmetric matrix $A$ of order $n \times n$, we use the following convention where the eigenvalues of A are in increasing order:
\begin{equation}\label{eqn:ev-hermitian}
\lambda_{\min}=\lambda_1(A) \leq \lambda_2(A)\leq \cdots \leq \lambda_{n-1}(A)\leq  \lambda_{n}(A)=\lambda_{\max}.
\end{equation}
We now state a few results from matrix theory  which are useful for subsequent results.
 
\begin{theorem}[Min-max Theorem]\cite{Horn}
Let $A$ be a  real symmetric matrix of order $n \times n$, and let the eigenvalues of $A$ be ordered as in Equation~\eqref{eqn:ev-hermitian}. Then 
$$\lambda_{\max}= \lambda_n(A)= \max_{\mathbf{x}^T \mathbf{x}=1} \mathbf{x}^T A \mathbf{x} = \max_{\mathbf{x}\neq \mathbf{0}} \frac{\mathbf{x}^T A \mathbf{x}}{\mathbf{x}^T \mathbf{x}} \quad  \mbox{  and  } \quad   \lambda_{\min}= \lambda_1(A)= \min_{\mathbf{x}^T \mathbf{x}=1} \mathbf{x}^T A \mathbf{x} = \min_{\mathbf{x}\neq \mathbf{0}} \frac{\mathbf{x}^T A \mathbf{x}}{\mathbf{x}^T \mathbf{x}}.$$
\end{theorem}
\begin{theorem}[Inclusion Principle]\cite{Horn}\label{thm:inclu}
Let $A$ be an $n\times n$ real symmetric matrix, let $r$ be an integer with $1\leq r\leq n$, and let $A_r$ denote any $r \times r$ principal submatrix of $A$. For each integer $k$ such that $1\leq k\leq r$, we have
$$\lambda_k(A)\leq \lambda_k(A_r)\leq \lambda_{k+n-r}(A).$$
\end{theorem}

\begin{theorem}~\cite[Theorem~$4.3.7$, Page $184$]{Horn}\label{thm:weyl}
Let $A$ and $B$ be real symmetric matrices of order $n \times n$ with eigenvalues ordered as in Equation~\eqref{eqn:ev-hermitian}. Then for every pair of integers $j,k$ such that $1\leq j,k\leq n$ and $j+k \geq n+1$, we have 
$$\lambda_{j+k-n}(A+B)\leq \lambda_j(A)+\lambda_k(B).$$
\end{theorem}

Let $T$ be a tree with positive weights on its edges. A  branch at a vertex $v$ of  $T$ is one of the connected components  obtained from $T$ by deleting $v$ and all edges incident with $v$. Let $L_v$ be the principal submatrix of the Laplacian matrix $L$ obtained by deleting the row and column corresponding to  the vertex $v$. It is easy to see that $L_v$ is a block diagonal invertible matrix. Hence $M_v=L_v^{-1}$ is a  block diagonal matrix and each of its  diagonal blocks corresponds to a branch at $v$, called  the bottleneck matrix  for that branch at $v$.  To be precise, for a branch $B$ at $v$ consisting  of $k$ vertices, the bottleneck matrix for $B$ based at $v$,  denoted by $M_v(B)$ is a $k \times k$ matrix such that for $x,y \in B$, the entry at the $(x,y)^{th}$ position of $M_v(B)$ is given by 
$$\ds \sum_{e \in \mathcal{P}(x,v)\, \cap \, \mathcal{P}(y,v)} \frac{1}{W(e)}.$$   If $u$ and $v$  are distinct vertices of a weighted  tree, we use $B_u(v)$ to denote the branch at the vertex $u$ which contains the vertex $v$. For notational convenience, we  write $M_v(u)$ instead of $M_v(B_v(u))$. Note that, the bottleneck matrix for a branch $B$ at $v$ is a square entrywise positive matrix, and the Perron value  of that bottleneck matrix $M_v(B)$ is $\rho(M_v(B))$,  called the Perron value of $B$. Finally, a branch $B$ at $v$ is called a Perron branch, if the Perron value of $B$ is the largest amongst all the branches at $v$ and hence  $\rho(M_v)=\rho(M_v(B)).$ 

The above notations and observations were first provided in~\cite{Kirkland, Kirkland1}. Let $T=(V,E)$ be tree.  For  any $u,v,w \in V$, we write $B_u(w)\subset B_v(w)$ if $B_u(w)$ contained in $B_v(w)$, and $B_u(w)\subsetneq B_v(w)$ if the containment is proper. The following result is a consequence of Theorems~\ref{thm:sp-comp} and \ref{thm:inclu}, and we state the result using the above notations.

\begin{prop}\label{prop:big-branch-scalar}
Let $T=(V,E)$ be a tree  with   positive weights on its edges. For  any $u,v,w \in V$, if $B_u(w)\subsetneq B_v(w)$, then $\rho(M_u(w))< \rho(M_v(w)).$
\end{prop} 

 We now recall an alternative characterization of the characteristic vertex and characteristic edge for trees with positive edge weights in terms of Perron branches and bottleneck matrices. As a consequence, a relation was found between the Perron values with the algebraic connectivity. The following results summarize these characterizations and some of their consequences (for details see~\cite{Kirkland}, \cite[Chapter $6$]{Molitierno}).

\begin{prop}\label{prop:ch-edge}
Let $T$ be a  tree with positive weights on its edges.  Then, the following statements are equivalent.
\begin{enumerate}
\item[$1.$] $T$ is Type II with the characteristic edge $e$ between the vertices $u$ and $v$.

\item[$2.$] There exists $0<\gamma<1$ such that  $\rho (M_u(v)-\gamma (1/\theta)J)=\rho (M_v(u)-(1-\gamma)(1/\theta)J)$, where $\theta$ is the weight of the edge $e$ between the vertices $u$ and $v$. Moreover,  $$\dfrac{1}{\mu(T)}=\rho (M_u(v)-\gamma (1/\theta)J)=\rho (M_v(u)-(1-\gamma)(1/\theta)J),$$ where $\mu(T)$ is the algebraic connectivity of $T$.

\item[$3.$]  For adjacent vertices $u$ and $v$,  $B_u(v)$ is the unique Perron branch  at $u$, while $B_v(u)$ is the unique Perron branch at $v$  in $T$.
\end{enumerate}
\end{prop}

\begin{prop}\label{prop:ch-vertex}
Let $T$ be a  tree with positive weights on its edges. Then, $T$ is Type I with the characteristic vertex $x$ if and only if there are two or more Perron branches of $T$ at $x$. Moreover,  the algebraic connectivity of $T$ is $1/\rho (M_x)$.
\end{prop}

\begin{prop}\label{prop:perron-branch}
Let $T$ be a tree with positive weights on its edges. If $x$ is not a characteristic vertex of $\, T$, then the unique Perron branch at $x$  in $T$ is the branch which contains the characteristic vertex (or vertices) of $T$.
\end{prop}

The above characterizations for trees have provided a new direction in understanding the structure of trees using the Laplacian matrix. In this direction, several intriguing results have been obtained by various researchers (for example, see~\cite{Abreu,Kirkland1, Lal, Patra1, Patra2}). In the last decade, some interesting results were obtained by considering graphs with matrix weights on their edges (for example, see~\cite{Atik,Atik1,Bapat1,Bapat2,Zhou}). Particularly, in~\cite{Atik}, the authors defined the  Laplacian matrix analogously for graphs with matrix weights on their edges.  As a special case, if the edge weights are positive definite matrices, then the Laplacian matrix is a positive semidefinite matrix. They have also proved an interesting result: Let $G$ be a connected graph on $n$ vertices with nonsingular matrix weights of order $s\times s$ on its edges and $L$ be the Laplacian matrix of $G$. Then,  the Laplacian matrix $L$ is of rank $(n-1)s$ if the graph $G$ is a tree.  However, the result is not necessarily true if the graph $G$ is not a tree. These developments have encouraged us to study the  Laplacian matrices of trees with matrix weights.

 In this article, our objective is to consider trees with a suitable class of matrix edge weights and establish the existence of some notion of the characteristic vertex (or vertices) using characterization in terms of Perron branches and Perron values analogous to trees with positive edge weights. We refer to such vertices as characteristic-like vertex (or vertices).  Moreover, we also provide a lower bound for the first non-zero eigenvalue of the Laplacian matrix. To be more specific, we are interested in trees with  the following  classes of matrix weights on their  edges:
\begin{enumerate}
\item[$1.$] positive definite matrix weights, 

\item[$2.$] lower (or upper) triangular matrix weights with positive diagonal entries.
\end{enumerate}

This article is organized as follows. In  Section~\ref{sec:Lap-n-bot}, we consider the principal submatrix $L_v$ of the Laplacian matrix  $L$ of a tree with matrix weights on its edges. We compute the determinant of $L_v$ and show that $L_v$ is an invertible matrix if and only if the edge weights are nonsingular matrices. Then, we find the inverse of $L_v$ and define the bottleneck matrix for a branch of a tree with nonsingular matrix edge weights. Further, using  $L_v^{-1}$, we find the Moore-Penrose inverse of the Laplacian matrix $L$. In Section~\ref{sec:char-perron}, we consider trees with the above class of matrix weights on their edges and show the existence of vertices satisfying properties analogous to the properties of characteristic vertices of trees with positive edge weights in terms of Perron values and Perron branches. Finally, in Section~\ref{sec:lower-bound}, we obtain a lower bound for the first non-zero eigenvalue of the Laplacian matrices of trees with the above classes of matrix edge weights in terms of Perron values.

\section{Laplacian Matrix and  Bottleneck Matrix }\label{sec:Lap-n-bot}

In this section, we consider the Laplacian matrices for trees with matrix weights on their edges and define the bottleneck matrix of a branch. As an application, we compute the Moore-Penrose inverse of the Laplacian matrix of a tree with nonsingular matrix weights.   

The Laplacian matrix  of a  graph with matrix weights on its edges  is  defined analogously. For the sake of completeness we recall its definition. Let $G=(V,E)$ be a graph on $n$ vertices and for each  edge $e \in E$ the associated matrix weight $W(e)$ is of order $s\times s$. The  Laplacian matrix $L(G)=[l_{uv}]$ is a  matrix of order $ns\times ns$ and is defined as follows:  for $u,v \in V$, if $u\neq v$,  then $l_{uv}$  is $\mathbf{0}$ if $u\nsim v$, and $l_{uv}$  is $-W(e)$ if $u\sim v$ and $e$ is the edge between them; finally if $u=v$, $l_{vv}$ is the sum of the weights of the edges in $G$ which are incident with the vertex $v$. We also write $L$ for the Laplacian matrix $L(G)$ if there is no scope for confusion. 

Before proceeding further, we recall the definition of the Kronecker product and some of its properties.
\begin{rem}\label{rem:rem1}
The Kronecker product of matrices $A=[a_{ij}]$ of order $m\times n$ and $B$ of order $p\times q$, denoted by $A\otimes B$, is defined to be the block matrix $[a_{ij}B]$. Then the following  hold true.
\begin{enumerate}
\item[$1.$] Let $A$ and $B$ be two square matrices. Let $\lambda\in \sigma(A)$ with corresponding eigenvector $\mathbf{x}$, and let $\mu\in \sigma(B)$ with corresponding eigenvector $\mathbf{y}$. Then $\lambda \mu$ is an eigenvalue of $A\otimes B$ with corresponding eigenvector $\mathbf{x} \otimes \mathbf{y}$. Moreover, any eigenvalue of $A\otimes B$  is a product of eigenvalues of $A$ and $B$.

\item[$2.$] Let $W$ be an $s\times s$ invertible matrix and  $\mathbf{y}_i^T \in \mathbb{R}^s$ for $1\leq i\leq n$.  Then, the vector  $\widetilde{\mathbf{y}}=(\mathbf{y}_1, \mathbf{y}_2, \ldots, \mathbf{y}_n)^T \in $ \textup{Null}($J_n \otimes W$) if and only if $\sum_{i=1}^n \mathbf{y}_i = \mathbf{0}. $
\end{enumerate}
\end{rem}
 
In~\cite{Atik}, it was shown that if $T$ is a tree on $n$ vertices with nonsingular matrix weights of order $s\times s$, then the rank of the Laplacian matrix $L$  of $T$ is $(n-1)s$.  Thus, it is natural to study the principal matrix $L_v$ of $L$ obtained by deleting the row block and column block corresponding to a vertex $v$. We begin with a result on the determinant of $L_v$. The proof is similar to that of trees with positive edge weights.

\begin{theorem}\label{thm:det-nonsingular}
Let $L$ be the Laplacian matrix  of a tree $T=(V,E)$ with  matrix weights on its edges. Let $L_{v}$ be the principal submatrix of $L$ obtained by deleting the row  block and column block corresponding to the vertex $v\in V$. Then 
$$\det L_{v}= \prod_{e\in E} \det W(e).$$
\end{theorem}
\begin{proof}
We prove this result by using induction on the number of vertices  $|V|=n$. The result is vacuously true  for $n=2$. Assume that the result is true for those trees whose number of vertices is strictly less than $n$. 

Let $v\in V$ and  $ \text{deg}(v)=r$. For $1\leq i\leq r$,  let $v$  be  adjacent to the vertex $v_i$ via the edge $e^{(i)}$. Thus, $B_v(v_i)$ for $1\leq i\leq r$ represents  all the branches at $v$ and  the block matrix $L_v$ can be written as 
\begin{equation}\label{eqn:prin-det}
{\small
L_v=\left[
\begin{array}{c |c| c| c}
    \widehat{L}(B_v(v_1)) & \mathbf{0}& \dots & \mathbf{0} \\
    \midrule
    \mathbf{0} & \widehat{L}(B_v(v_2)) & \dots & \mathbf{0} \\
    \midrule    
    \vdots & \vdots & \ddots & \vdots \\
    \midrule
    \mathbf{0} & \mathbf{0}& \dots & \widehat{L}(B_v(v_r))
\end{array}
\right],}
\end{equation}
where $\widehat{L}(B_v(v_i))$ is the principal submatrix of $L$ corresponding to the branch $B_v(v_i)$ for $1\leq i\leq r$.  

Let $L(B_v(v_i))$ denote the Laplacian matrix of the branch $B_v(v_i)$. Then
$$\widehat{L}(B_v(v_i))= L(B_v(v_i)) + \mathbf{e}_{v_i} \mathbf{e}_{v_i}^T\otimes W(e^{(i)}) \mbox{ for } 1\leq i\leq r, $$ 
where $\mathbf{e}_{v_i}$ is the column vector of  conformal order with $1$ at $v_i^{th}$ entry and $0$ elsewhere.

For each $1\leq i\leq r$,  let $L(B_v(v_i))_{v_i}$ denote the principal submatrix  of $ L(B_v(v_i))$ obtained by deleting the row block and column block  corresponding to the vertex $v_i$.  Using the induction hypothesis, we have 
\begin{equation}\label{eqn:prin-det1}
\det L(B_v(v_i))_{v_i}= \prod_{e\in E(B_v(v_i))} \det W(e).
\end{equation}
Further, if  we  add all the column blocks  of  $\widehat{L}(B_v(v_i))$  to the column block corresponding to the vertex $v_i$ and repeat a similar operation for row blocks, then the resulting matrix  can be represented as  
{\small $$\left[
\begin{array}{c|c}
L(B_v(v_i))_{v_i}& \mathbf{0} \\
\midrule
\mathbf{0} & W(e^{(i)})
\end{array}
\right].$$}
Using Equation~\eqref{eqn:prin-det1}, we now get
$$\det \widehat{L}(B_v(v_i))= \det W(e^{(i)})  \times \prod_{e\in E(B_v(v_i))} \det W(e).$$
The desired result follows from Equation~\eqref{eqn:prin-det}.
\end{proof}

Under the hypothesis of the above theorem,  if $T$ is a tree with nonsingular matrix edge weights, then $L_v$ is an invertible matrix. Our next objective is to find the inverse of $L_v$. We first consider the case where the edge weights are positive definite matrices. We now recall the definition of the incidence matrices of graphs with positive definite matrix edge weights.

Let $G=(V, E)$ be a   graph with $n$ vertices and $m$ edges such that the weights associated with the edges are positive definite matrices of order $s\times s$.   We assign an orientation to each edge of $G$. Then, the vertex-edge incidence matrix $Q$ is a block matrix such that the row blocks are indexed by the vertex set $V$ and the column blocks are indexed by the edge set $E$.   The  vertex-edge incidence matrix $Q= [Q_{ue}]$ is a matrix of order $ns \times ms$, where
\begin{equation}\label{eqn:inci-matrix}
Q_{ue}= \left\{\begin{array}{rl}  \sqrt{W(e)} & \mbox{if}\ u \ \mbox{is the initial vertex of the edge e}, \\
              -\sqrt{W(e)} &  \mbox{if}\ u \  \mbox{is the terminal vertex of the edge e},\\
                  \mathbf{0} &  \mbox{otherwise}. \end{array}\right.
\end{equation}

It can be seen that, for a given graph $G=(V, E)$ with positive definite weights on its edges,  the Laplacian matrix $L$ of $G$  is given by $L=QQ^{T}$. Then  $L_{v}= Q_v Q_v^{T}$, where $Q_v$  is the block matrix obtained by deleting the row block of $Q$ corresponding to the vertex $v \in V$. In particular, if $G$ is a tree, then by  Theorem~\ref{thm:det-nonsingular} we have $\det L_v= (\det Q_v)^2\neq 0$. This implies that $Q_v$ is an invertible matrix and $L_v^{-1}= (Q_v^{-1})^T Q_v^{-1}$. 
We now compute the inverse of $Q_v$ when the graph $G$ is a tree. The argument used to find the inverse $Q_v^{-1}$ is similar to the proof for those trees whose edge weights are all $1$ (for details, see~\cite[Chapter~$2$]{Bapat}). 

 Given a path $\mathscr{P}$ in $G$, the incidence block vector of $\mathscr{P}$  is an $ms\times s$ matrix   (a column block indexed by the edge set $E$) and is defined as follows: for any $e \in E$, the  entry corresponding to $e$ is the matrix $\mathbf{0}$, if the path does not contain $e$. If the path contains $e$, then the entry corresponding to $e$ is $\left(\sqrt{W(e)}\right)^{-1}$ or $-\left(\sqrt{W(e)}\right)^{-1}$, depending on whether the direction of the path agrees or disagrees, respectively with $e$.

Let $T=(V,E)$ be a tree. For $v\in V$,  the path matrix $P_v$ of $T$ is an $ms\times (n-1)s$ matrix ($P_v$ is a block matrix such that rows are indexed by the edge set $E$  and column blocks indexed  by the  vertex set $V-\{v\}$) and is defined  as follows. For $u \in  V-\{v\}$, the  column block corresponding to the vertex $u$ of $P_v$ is the incidence vector of the path from $u$ to $v$. 

\begin{theorem}\label{thm:pd-bt}
Let $T=(V,E)$ be a tree with  positive definite matrix  weights on its edges. Let $Q$ be the incidence matrix of $T$ and  $Q_v$ be the block matrix obtained by deleting the row block of $Q$ corresponding to the vertex $v \in V$. Then $Q_v^{-1}= P_v$.
\end{theorem}
\begin{proof}
 Let $X=P_v Q_v$ be an $ms\times ms$  matrix. For $i\neq j$,  let $e_i$ be the edge from $x$ to $y$ and let  $e_j$ be the edge from $w$ to $z$. The $(u,e_j)^{th}$ entry of the  incidence matrix $Q$ is  $Q_{ue_j}=\mathbf{0}$ unless $u=w$ or $u=z$. Thus, 
$$\ds X_{e_ie_j}= \sum_{u\in V-\{v\}} P_{e_i u}Q_{ue_j}= P_{e_iw}Q_{we_j}+ P_{e_iz}Q_{ze_j}=( P_{e_iw}-P_{e_iz})\sqrt{ W(e_j)}.$$  
Note that, the path from $w$ to $v$ contains $e_i$ if and only if the path from $z$ to $v$ contains $e_i$. Moreover, if $P_{e_iw}$ and $P_{e_iz}$ are non-zero, then they share the same sign. Thus, $P_{e_iw}=P_{e_iz}$ which implies that $\ds X_{e_ie_j}=\mathbf{0}$ whenever $i \neq j$.

For $i=j$, the path from $x$ to $v$ contains $e_i$ if and only if   the path from $y$ to $v$ does not contain $e_i$. Thus, if $e_i$ is in the path from $x$ to $v$, then $ X_{e_ie_i}= P_{e_ix}Q_{xe_i}=\left(\sqrt{ W(e_j)}\right)^{-1} \sqrt{ W(e_j)} =I$. Similarly, if $e_i$ is in the path from $y$ to $v$, then $ X_{e_ie_i}= P_{e_iy}Q_{ye_i}=\left(- \left(\sqrt{ W(e_j)}\right)^{-1}\right)\left(- \sqrt{ W(e_j)}\right) =I$. This completes the proof.
\end{proof}

\begin{cor}\label{cor:inv}
Let $T=(V,E)$ be a tree with  positive definite matrix  weights  on its edges and   $L$ be the Laplacian matrix of $T$.  Let $L_{v}$ denote the principal submatrix of $L$ obtained by deleting the row  block and column block corresponding to  the vertex $v\in V$. Then for $u,w\in V-\{v\}$, the block at  $(u,w)$ position   of $L_v^{-1}$ is given by
$$(L_v^{-1})_{uw}= \sum_{e \in \mathcal{P}(u,v)\, \cap \, \mathcal{P}(w,v)} W(e)^{-1},$$
where $\mathcal{P}(x,y)$ denotes the path joining the vertices   $x$  and $y$ in $T$.
\end{cor}
\begin{proof}
Using $L_{v}= Q_v Q_v^{T}$ and by  Theorem~\ref{thm:pd-bt},  we  have  $L_v^{-1}= P_v^T P_v$. Thus
$$(L_v^{-1})_{uw}= \sum_{e\in E} P_{eu}P_{ew}.$$
Further, $P_{eu}P_{ew}$ is non-zero  if and only if  the edge $e$ is in both the paths $\mathcal{P}(u,v)$ and $\mathcal{P}(w,v)$. In this case, the orientation of $e$ either agrees or disagrees simultaneously, for both  paths. Thus, $P_{eu}P_{ew}=W(e)^{-1}$, if $e \in \mathcal{P}(u,v)\, \cap \, \mathcal{P}(w,v) $ and    $\mathbf{0}$ otherwise. Hence the result follows.
\end{proof}

By Corollary~\ref{cor:inv}, the block matrix form of $L_v^{-1}$  for trees with positive definite matrix weights on its edges is similar to the case where trees have positive edge weights (for details, see~\cite[Proposition $1$]{Kirkland}). We now show that the above block matrix form is unchanged even if the weights assigned to the edges are nonsingular.

\begin{theorem}\label{thm:non-sing-inv}
Let $L$ be the Laplacian matrix  of a tree $T=(V,E)$ with nonsingular matrix weights  on its edges. Let $L_{v}$ denote the principal submatrix of $L$ obtained by deleting the row  block and column block corresponding to  the vertex $v\in V$. Then for $u,w\in V-\{v\}$, the block at  $(u,w)$ position   of $L_v^{-1}$ is given by
$$(L_v^{-1})_{uw}= \sum_{e \in \mathcal{P}(u,v)\, \cap \, \mathcal{P}(w,v)} W(e)^{-1},$$
where $\mathcal{P}(x,y)$ denotes the path joining the vertices   $x$  and $y$ in $T$.
\end{theorem}
\begin{proof}
For $u,w \in V -\{v\},$ let  $B=[B_{uw}]$ be an $(n-1)s\times (n-1)s$ block matrix, where
$$\ds B_{uw}=\sum_{e \in \mathcal{P}(u,v)\, \cap \, \mathcal{P}(w,v)} W(e)^{-1}.$$ 
Let $L=[l_{xy}]_{x,y\in V}$  and $X=L_vB=[X_{uw}]$. Then, for $u,w\in V -\{v\}$, we have 
\begin{equation}\label{eqn:Xuw}
X_{uw}= l_{uu}B_{uw}+ \sum_{x\sim u \atop x\neq v} l_{ux}B_{xw}.
\end{equation}
For a given $u \in V -\{v\} $, let $\text{deg}(u)=r$. For $1\leq i\leq r$, let $u$ be   adjacent to the vertex $v_i$ via the edge $e^{(i)}$. We  consider the cases $u=w$ and $u\neq w$ below.\\

\noindent$\underline{\textbf{Case 1:}}$ For $u=w$. \\

If $u=w$ and $u\nsim v$, then  the path $ \mathcal{P}(u,v)$ contains exactly one vertex adjacent to $u$. Without loss of generality, let $v_r \in \mathcal{P}(u,v)$. Then $B_{v_iu}=B_{uu}= \sum_{e \in \mathcal{P}(u,v)} W(e)^{-1}$ for $1\leq i\leq r-1$, and $B_{v_ru}=B_{uu}- W(e^{(r)})^{-1}$. Using Equation~\eqref{eqn:Xuw}, we have $X_{uu}=l_{uu}B_{uu}+  l_{u v_r}B_{v_r u}+\sum_{i=1}^{r-1} l_{uv_i}B_{v_i u}= \left(l_{uu}+\sum_{i=1}^{r} l_{uv_i} \right) B_{uu} + l_{u v_r} \left(- W(e^{(r)})^{-1}\right)$. Since the  row  block sum of $L$  is zero, {\it i.e.,} $l_{uu}+\sum_{i=1}^{r} l_{uv_i} =\mathbf{0}$ and $l_{uv_r}= -  W(e^{(r)})$, so  $X_{uu}=I$. 

If $u=w$ and $u\sim v$, then $v=v_r$ and   $B_{v_i u}=B_{uu}=W(e^{(r)})^{-1}$ for $1\leq i\leq r-1$. Using Equation~\eqref{eqn:Xuw}, we have $ X_{uu}=l_{uu}B_{uu}+\sum_{i=1}^{r-1} l_{uv_i}B_{v_i u}
      = \left( l_{uu}+\sum_{i=1}^{r-1} l_{uv_i} \right)  W(e^{(r)})^{-1}
      = \left(- l_{u v_r} \right)  \left(- W(e^{(r)})^{-1}\right) =  W(e^{(r)})  W(e^{(r)})^{-1}= I.$\\
      
\noindent $\underline{\textbf{Case 2:}}$ For $u\neq w$.  We consider the following sub cases to complete the proof.\\

$\underline{\textbf{Subcase 2.1:}}$  For $u\neq w$ and $u\nsim v$.\\

 If $u\in \mathcal{P}(w,v)$,  then  both the paths $\mathcal{P}(w,u)$ and $ \mathcal{P}(u,v) $ contain exactly one vertex each, which are adjacent to $u$.  Without loss of generality, let $ v_1 \in \mathcal{P}(w,u) $ and $v_r \in \mathcal{P}(u,v)$. Then   $B_{v_1w}=B_{uw}+  W(e^{(1)})^{-1}$, $B_{v_rw}=B_{uw}-  W(e^{(r)})^{-1}$ and $B_{v_iw}=B_{uw}$ for $2\leq i\leq r-1$. By Equation~\eqref{eqn:Xuw}, we have $X_{uw}= \left( l_{uu}+\sum_{i=1}^{r} l_{uv_i} \right) B_{uw} + l_{u v_1} \left( W(e^{(1)})^{-1}\right)+ l_{u v_r} \left(- W(e^{(r)})^{-1}\right)=\mathbf{0}+ \left(- W(e^{(1)})\right) \left( W(e^{(1)})^{-1}\right)+ \left(- W(e^{(r)})\right) \left(- W(e^{(r)})^{-1}\right) = -I +I=\mathbf{0}.$

If $u\notin \mathcal{P}(w,v)$, then either 
$\mathcal{P}(u,v) \cap \mathcal{P}(w,v)=\emptyset $  or $w \in \mathcal{P}(u,v)$. Note that, if $\mathcal{P}(u,v) \cap \mathcal{P}(w,v)=\emptyset $ and $u\notin \mathcal{P}(w,v)$, we have $B_{v_iw}=B_{uw}=\mathbf{0}$ for $1\leq i\leq r$, then $X_{uw}=\mathbf{0}$.  For $w \in \mathcal{P}(u,v)$, we have  $B_{v_iw}=B_{uw}=\sum_{e \in \mathcal{P}(w,v)} W(e)^{-1}$ and  Equation~\eqref{eqn:Xuw} yields $X_{uw}= \left( l_{uu}+\sum_{i=1}^{r} l_{uv_i} \right) B_{uw}=\mathbf{0}$.\\

$\underline{\textbf{Subcase 2.2:}}$  For $u\neq w$ and $u\sim v$.\\ 
 
Let $v=v_r$.  If $u\in \mathcal{P}(w,v)$,  then   the path $\mathcal{P}(w,u)$ contains exactly one vertex $v_1$ (say) adjacent to $u$. In that case, $B_{v_1w}=W(e^{(r)})^{-1}+  W(e^{(1)})^{-1}$ and $B_{v_iw}=B_{uw}=W(e^{(r)})^{-1}$ for $2\leq i\leq r-1$. Using Equation~\eqref{eqn:Xuw}, we get $X_{uw}= \left( l_{uu}+\sum_{i=1}^{r-1} l_{uv_i} \right)  W(e^{(r)})^{-1}+ l_{v_1w}W(e^{(1)})^{-1} = \left(- l_{u v_r} \right)  \left(- W(e^{(r)})^{-1}\right)- I=\mathbf{0}$. If $u\notin \mathcal{P}(w,v)$, then  
$\mathcal{P}(u,v) \cap \mathcal{P}(w,v)=\emptyset $, which implies that  $B_{v_iw}=B_{uw}=\mathbf{0}$ for $1\leq i\leq r-1$ and hence the result follows.
\end{proof}

In view of  Theorems~\ref{thm:det-nonsingular} and \ref{thm:non-sing-inv},  for trees with nonsingular matrix edge weights, we define the bottleneck matrix of a branch as follows. Let $L$ be the Laplacian matrix  of a tree $T=(V,E)$ with nonsingular matrix weights of order $s\times s$ on its edges and  $L_{v}$ be the principal submatrix of $L$ obtained by deleting the row  block and column block corresponding to  a vertex $v\in V$. Let  $\deg(v)=r$ and for $1\leq i\leq r$, let  $B_i$ be the branches  at $v$ in $T$. By Theorem~\ref{thm:det-nonsingular}, $L_{v}$ is an invertible matrix. Denote $L_v^{-1}$ by $M_v$. Then 
 \begin{equation*}
{\small
L_v=\left[
\begin{array}{c |c| c| c}
    \widehat{L}(B_1) & \mathbf{0}& \dots & \mathbf{0} \\
    \midrule
    \mathbf{0} & \widehat{L}(B_2) & \dots & \mathbf{0} \\
    \midrule    
    \vdots & \vdots & \ddots & \vdots \\
    \midrule
    \mathbf{0} & \mathbf{0}& \dots & \widehat{L}(B_r)
\end{array}
\right]}
 \mbox{ and }
{\small
M_v=L_v^{-1}=\left[
\begin{array}{c |c| c| c}
    M_v(B_1) & \mathbf{0}& \dots & \mathbf{0} \\
    \midrule
    \mathbf{0} & M_v(B_2) & \dots & \mathbf{0} \\
    \midrule    
    \vdots & \vdots & \ddots & \vdots \\
    \midrule
    \mathbf{0} & \mathbf{0}& \dots & M_v(B_r)
\end{array}
\right],}
\end{equation*}
where $ M_v(B_i) = \widehat{L}(B_i)^{-1}$ is called the bottleneck matrix of the branch $B_i$  at $v$ in  $T$ for $1\leq i\leq r$. Thus, by Theorem~\ref{thm:non-sing-inv}, 
for a branch $B$ at $v$ consisting of $k$ vertices, the bottleneck matrix $M_v(B)$ for $B$ based at $v$ is a $ks \times ks$   matrix such that ($M_v(B)$ as a block  matrix) for $x,y \in B$, the block at the  $(x,y)^{th}$ position   of $M_v(B)$ is given by 
$$\ds \sum_{e \in \mathcal{P}(x,v)\, \cap \, \mathcal{P}(y,v)} W(e)^{-1}.$$ 
As in the case of trees with positive edge  weights,   for notational convenience, we  write $M_v(u)$ instead of $M_v(B_v(u))$,  whenever $B_v(u)$ is  the branch  at $v$  in $T$ containing the vertex $u$.

In this manuscript, we consider a few classes of matrix weights on the edges of $T$ so that the  eigenvalues of $L_v$ are  positive. Therefore, for each  $1\leq i\leq r$,  all the eigenvalues of the bottleneck matrices  $ M_v(B_i)$ are positive and the spectral radius of  $ M_v(B_i)$ is necessarily   an eigenvalue. In this case, the spectral radius of  $ M_v(B_i)$ need not be a simple eigenvalue, but  continuing with the terminology  as in the case of the trees  with positive  weights on edges,  we call the spectral radius $\rho( M_v(B_i))$  as the Perron value of  the bottleneck matrix $M_v(B_i) $ for  $1\leq i\leq r$. Thus, the spectrum of $M_v=L_v^{-1}$ is given by $\sigma(M_v)= \bigcup_{i=1}^r \sigma(M_v(B_i))$ and the spectral radius of $M_v$ is given by
\begin{equation}\label{eqn:sp-radius}
\rho(M_v)= \max_{1\leq i\leq r} \rho( M_v(B_i)).
\end{equation}
 We also define the Perron value of a branch  at $v$  in $T$ as the Perron value of the corresponding bottleneck matrix (or matrices) for which the maximum is attained. We call such a branch at $v$ as a Perron branch if the Perron value of that branch is the same as the spectral radius of $L^{-1}_v$.

Finally, we conclude this section with a few results that allow us to compute the Moore-Penrose inverse of the Laplacian matrices for trees with nonsingular matrix edge weights. We obtain this result as an application of Theorem~\ref{thm:non-sing-inv}.

If $A$ is  an $m\times n $ matrix, then an $n\times m$ matrix $\Gamma$ is called a generalized  inverse of $A$ if $A\Gamma A=A$. The Moore-Penrose inverse of $A$, denoted by $A^+$, is an $m\times n$ matrix satisfying the following equations:
$$AA^{+}A=A,\quad A^{+}AA^{+}=A^{+}, \quad (AA^{+})^T=AA^{+},  \quad (A^{+}A)^T=A^{+}A.$$
It is well known that any complex matrix admits a unique Moore-Penrose inverse, and we refer to~\cite{Ben,Campbell} for basic properties of the Moore-Penrose inverse. One such property is that the null space of $A^+$ is the same as that of $A^T$ for any matrix $A$, and we present this result as a lemma without proof.

\begin{lem}\label{lem:Moore-null}
If $A$ is  an $m\times n $ matrix, then  for any $n\times 1$ vector $\mathbf{x}$, $A\mathbf{x}=\mathbf{0}$ if and only if $\mathbf{x}^T A^+= \mathbf{0}.$
\end{lem} 

The following result was proved for connected graphs with positive definite matrix edge weights (see~\cite[Theorem~$3.4$]{Atik1}). We now show that the result holds for trees with nonsingular matrix edge weights.

\begin{lem}\label{lem:1-ll+}
Let $L$ be the Laplacian matrix  of a tree $T$ on $n$ vertices with nonsingular matrix weights of  order $s\times s$ on its edges. Then $I_{ns}-LL^{+}= J_n \otimes \frac{1}{n} I_s$.
\end{lem}

\begin{proof}
Let $I_{ns}-LL^{+}=[X_{ij}]$, where  each $X_{ij}$  is a matrix of order  $s \times s$. Since $(I-LL^{+})L=\mathbf{0}$,  each row of $I-LL^{+}$  belongs to the left  null space of $L$. Recall that, the row and column block sum of $L$ is zero and also the null space is of dimension $s$. Thus, the rows of $\mathds{1}_n^{T} \otimes I_s$ generate the left null space, which implies that,  any row of $I_{ns}-LL^{+}$ is of the form $\mathds{1}_n^{T} \otimes \mathbf{x}$, for some $\mathbf{x} \in \R^{n}$. Since $I_{ns}-LL^{+}=[X_{ij}]$ is symmetric,  by the above argument we get $X_{ij}=X$ for  $1\leq i,j\leq n$ and hence $I_{ns}-LL^{+}=J_n \otimes X$.

Note that, Null($L^{+}$) $\subset$  Range($I_{ns}-LL^{+}$), this  implies that $ \textup{rank}(I_{ns}-LL^{+}) \geq \textup{nullity} (L^{+})= \textup{nullity} (L)=s.$  Since  $J_n$ is a rank one matrix,  $\textup{rank}(I_{ns}-LL^{+})= \textup{rank}(J_n \otimes X)= \textup{rank}(X)\leq s$. Thus, $\textup{rank}(I_{ns}-LL^{+})= \textup{rank}(J_n \otimes X)=\textup{rank}(X)= s$. Hence, $X$ is a nonsingular matrix. Since $I_{ns}-LL^{+}$ is idempotent, 
$$I_{ns}-LL^{+}=J_n \otimes X= (J_n \otimes X)^2= J_n^2 \otimes X^2= n J_n \otimes X^2,$$
which implies that $X=nX^2$. Since $X$ is a nonsingular matrix,  $X=\frac{1}{n} I_s$. This completes the proof.
\end{proof}

We now state a lemma that is useful in computing the Moore-Penrose inverse of the Laplacian matrix. The result is easy to verify, and hence the proof is omitted.
 
\begin{lem}\label{lem:M}
If $M=I_{(n-1)s}- J_{n-1} \otimes \frac{1}{n} I_s$, then $M^{-1}= I_{(n-1)s}+ D D^{T}$, where   $D=\mathds{1}_{n-1}\otimes I_s.$ 
\end{lem} 

We now compute the Moore-Penrose inverse $L^{+}$ of the Laplacian matrix $L$ for trees with nonsingular matrix weights.
\begin{theorem}
Let $T=(V,E)$ be a tree on $n$ vertices such that the weights  associated with the edges are nonsingular matrices of  order $s\times s$  and  $L$ be the Laplacian matrix  of  $T$. Let $L_v$ be the principal submatrix of $L$ and $(L^{+})_v$ be the principal submatrix of $L^{+}$, obtained by deleting the row and column corresponding to the vertex $v \in V$. Then $(L^{+})_v= M L^{-1}_v M,$ where $M=I_{(n-1)s}- J_{n-1} \otimes \frac{1}{n} I_s$. Moreover, if $V=\{v_1,v_2,\ldots, v_n\}$ is the  ordering  of vertices in the Laplacian matrix $L$ and $v=v_n$, then 
$$L^{+}=\left[
\begin{array}{c}
\mathbb{X}\\
\hline
\mathbb{Y}
\end{array}
\right],$$
where $\mathbb{X}=\left[
\begin{array}{c | c}
(L^{+})_{v_n} & (L^{+})_{v_n} (\mathds{1}_{n-1}\otimes I_s) 
\end{array}
\right]$ and $\mathbb{Y}= (\mathds{1}_{n-1}^T\otimes I_s)\mathbb{X}.$
\end{theorem}
\begin{proof}
Let $V=\{v_1,v_2,\ldots, v_n\}$ be the  ordering of the vertices in the Laplacian matrix $L$ and without loss generality, let  $v=v_n$. Let us partition the Laplacian matrix  as    $L=\left[
\begin{array}{c | c}
\mathbb{U}& \mathbb{V} 
\end{array}
\right], $
where $\mathbb{V}$  represents the  column block corresponding to the vertex $v_n$. Also partition the Moore-Penrose inverse $L^{+}$ as $L^{+}=\left[
\begin{array}{c}
\mathbb{X}\\
\hline
\mathbb{Y}
\end{array}
\right],$
where $\mathbb{Y}$  represents the  row block corresponding to the vertex $v_n$. Since column block sum of $L$ is zero,  $\mathbb{V}= D\mathbb{U}$, where $D=\mathds{1}_{n-1}\otimes I_s.$ Thus $\mathbb{Y}= D^T\mathbb{X}$. Using Lemma~\ref{lem:1-ll+}, we have
$$I_{ns}- J_n \otimes \frac{1}{n} I_s=LL^{+}= \mathbb{U}\mathbb{X}+\mathbb{V}\mathbb{Y}= \mathbb{U}(I_{(n-1)s}+ D D^{T})\mathbb{X}. $$
Thus,
$M=\mathbb{U}_n(I_{(n-1)s}+ D D^{T})\mathbb{X}^{n}$, 
where $\mathbb{U}_n$ is the matrix obtained by deleting the row block corresponding to the  vertex $v_n$ and $\mathbb{X}^{n}$ is the matrix obtained by deleting the column block corresponding to the vertex $v_n$.  Since $\mathbb{U}_n = L_{v_n}$ and $\mathbb{X}^{n}=(L^{+})_{v_n}$,  by Lemmas~\ref{thm:det-nonsingular} 
and~\ref{lem:M}, we have
$$(L^{+})_{v_n}=\mathbb{X}^{n}= (I_{(n-1)s}+ D D^{T})^{-1}\mathbb{U}_n^{-1} M= M L^{-1}_{v_n} M.$$
In view of Lemma~\ref{lem:Moore-null},  the column block  corresponding to the vertex $v_n$ of  $\mathbb{X}$ is $ (L^{+})_{v_n} D$ and  $\mathbb{Y}= D^T \mathbb{X}.$ This completes the proof.
\end{proof}

\section{Characteristic-like Vertices and Perron Values} \label{sec:char-perron}

 In this section, we consider  trees with  the following  classes of matrix weights on their edges:
\begin{enumerate}
\item[$1.$] positive definite matrix weights, 

\item[$2.$] lower (or upper) triangular matrix weights with positive diagonal entries.
\end{enumerate}
Using Perron branches, we show the existence of vertices with properties analogous to characteristic vertices of trees with positive edge weights as stated in  Propositions~\ref{prop:ch-edge} -~\ref{prop:perron-branch}. We call such vertices characteristic-like vertices. To be more precise, our objective in this section is to prove the following results.

\begin{result}\label{result:1}
Let $T$ be a tree with either of the following classes of matrix weights on its edges: $(1)$~positive definite matrix weights,  $(2)$ lower (or upper) triangular matrix weights with positive diagonal entries. Then one of the following cases occurs: 
\begin{enumerate}
\item[$1.$] There is a unique vertex $v$ such that there are two or more   Perron branches at $v$ in $T$.  

\item[$2.$] There is  a unique  pair of vertices $u$ and $v$ with $u \sim v$  such that   the Perron branch  at $u$ in $T$ is  the branch containing $v$, while the Perron branch  at $v$  in $T$  is  the branch containing $u$. 
\end{enumerate}
 \end{result}

It is easy to see that if Result~\ref{result:1} is true, then it allows us to define a notion analogous to the characteristic vertex and the characteristic edge for trees with the above classes of matrix edge weights using Perron branches. We now formally define the characteristic-like vertex and the characteristic-like edge on trees with the above-mentioned classes of matrix weights on their edges.

\begin{defn}\label{def:ch-vertices}
Let $T$ be a tree with either of the following classes of matrix weights on its edges: $(1)$~positive definite matrix weights,  $(2)$ lower (or upper) triangular matrix weights with positive diagonal entries. Then  one of the following cases occurs: 
\begin{enumerate}
\item[$1.$] There is a unique vertex $v$ such that there are two or more   Perron branches at $v$ in $T$.  In this case,  the vertex $v$ is called the characteristic-like vertex of $T$.

\item[$2.$] There is a unique pair of vertices $u$ and $v$  with $u \sim v$ such that the Perron branch at $u$ in $T$ is the branch containing $v$, while the Perron branch at $v$ in $T$ is the branch containing $u$. In this case, we call the edge between the vertices $u$ and $v$,  the characteristic-like edge of $T$.
\end{enumerate}
\end{defn}
 It is easy to see that the notions of characteristic-like vertex and characteristic-like edge coincide with the notions of characteristic vertex and characteristic edge for trees with positive edge weights. Let $T$ be a tree with either of the above-mentioned matrix weights on its edges and let $\mathcal{C}_T$  denote the set of characteristic-like vertices of $T$. Then, $|\mathcal{C}_T|=1$ or $2$,  depending on whether $T$ contains a characteristic-like vertex or characteristic-like edge, respectively. We now state the second desired result.
  
\begin{result}\label{result:2}
Let $T$ be a tree with either of the following classes of matrix weights on its edges: $1.$~positive definite matrix weights,  $2.$ lower (or upper) triangular matrix weights with positive diagonal entries.   If $x$ is not a characteristic-like vertex of $T$, then the unique Perron branch at $x$ in $T$ is the branch which contains the characteristic-like vertex (or vertices) of $T$.
\end{result}

\subsection{Results for Positive Definite Matrix Weights}

In this section, we consider trees with positive definite matrix weights on their edges.  Let $L$ be the Laplacian matrix of a tree $T$ with positive definite matrix weights on its edges and  $L_{v}$ be the principal submatrix of $L$ formed by deleting the row block and column block corresponding to the vertex $v$. By the previous section, we know that $L=Q Q^T$ is a positive semidefinite matrix. In view of Theorems~\ref{thm:inclu} and~\ref{thm:det-nonsingular},  $L_{v}$  is a positive definite matrix. Therefore, the definitions of the Perron value and the  Perron branch (as defined in Section~\ref{sec:Lap-n-bot}) are well-defined for trees with positive definite matrix edge weights. Before proceeding further, we prove a lemma which is useful in our subsequent proofs.

\begin{lem}\label{lem:big-branch-sr}
Let $T=(V,E)$ be a tree  with positive definite  matrix weights  on its edges. For  any $u,v,w \in V$, if $B_u(w)\subset B_v(w)$, then $\rho(M_u(w))\leq \rho(M_v(w)).$
\end{lem}  
\begin{proof}
For any $u,v,w \in V$,  we have $B_u(w)\subset B_v(w)$.  Thus, by renaming the vertices, the matrix $M_v(w)$ can be written as
 $${\small 
 M_v(w)=\left[
\begin{array}{c|c}
M_u(w)+ J\otimes W & \ast \\
\midrule
\ast & \ast
\end{array}
\right]},$$
where $W= \sum_{e \in \mathcal{P}(u,v)}W(e)^{-1}$ is a positive definite matrix. Using Theorem~\ref{thm:inclu}, we have $\rho(M_v(w) )\geq \rho(M_u(w)+J\otimes W)$. Since $J\otimes W$ is a positive semidefinite matrix,  using the min-max theorem  we get  $\rho(M_u(w)+J\otimes W) \geq \rho(M_u(w))$. This completes the proof.
\end{proof}

Next, we establish the existence of vertices with properties analogous to characteristic vertices in terms of Perron branches. We first show the existence of a characteristic-like edge for trees with positive definite matrix weights.

\begin{theorem}\label{thm:ch-edge1}
Let $T$ be a tree with positive definite matrix weights on its edges such that there is a unique Perron branch at every vertex in $T$. Then,  there is a unique pair of vertices $u$ and $v$ in $T$ with $u \sim v$  such that the Perron branch at $u$ is the branch containing $v$,  while the Perron branch at $v$ is the branch containing $u$. Moreover, the unique Perron branch at any vertex $x$ in $T$ is the branch which contains at least one of the vertices $u$ or $v$.
\end{theorem}
\begin{proof}
Let $u_1$ be a vertex in $T$.  For $u_2\sim u_1$, let $B_{u_1}(u_2)$ be the unique Perron branch  at $u_1$ in $T$. Proceeding inductively we find  a walk $u_1\sim u_2\sim u_3\sim \cdots$ such that  $B_{u_i}(u_{i+1})$ is the  unique Perron branch  at $u_i$ in $T$ for $i=1,2,\ldots$. Since  $T$ is acyclic and finite,  there exists $i_0$ such that $u_{i_0+1}=u_{i_0-1}$, {\it i.e.,}  $B_{u_{i_0}}(u_{i_0-1})$ is the  unique  Perron branch at $u_{i_0}$,  whereas $B_{u_{i_0-1}}(u_{i_0})$ is the unique Perron branch at $u_{i_0-1}$ in $T$.  Let us denote $u_{i_0-1}$ by $u$ and $u_{i_0}$ by $v$. Thus, the Perron branch at $u$ is  the branch $B_u(v)$  and the Perron branch at $v$ is  the branch $B_v(u)$.  

We prove the uniqueness  of  the vertices $u$ and $v$ as follows. Let $x,y$ be any two vertices in $B_u(v)$ such that $x\sim y$ and $y \notin B_x(u)$. Thus  $B_x(y)\subseteq B_v(y)$ and  $B_v(u) \subseteq B_x(u)$. Using Lemma~\ref{lem:big-branch-sr}, we get
\begin{equation}\label{eqn:ch-edge}
\rho(M_x(y))\leq \rho(M_v(y))  \mbox{ and }  \rho(M_v(u))\leq \rho(M_x(u)).
\end{equation}
Since the unique Perron branch  at $v$ in $T$ is $B_v(u)$, we see that $ \rho(M_v(y))  < \rho(M_v(u)).$ Therefore, by Equation~\eqref{eqn:ch-edge}, we have $\rho(M_x(y))<  \rho(M_x(u))$, which implies that the unique Perron branch  at $x$ in $T$ is the branch which contains at least one of the  vertices $u$ or $v$. Similar assertions can be made whenever we consider $x,y \in B_v(u)$. Hence the result follows.
\end{proof}

Before proving the existence of characteristic-like vertex for trees with positive definite matrix weights, we prove the following lemma.

\begin{lem}\label{lem:ch-ver-lem-pd}
Let $T$ be a tree  with positive definite  matrix weights  on its edges. If there exists a vertex $v$ in $T$ such that  there are two or more  Perron branches   at $v$ in $T$, then for any vertex $x$  other than $v$, the branch $B_x(v)$ is a Perron branch  at $x$ in $T$, {\it i.e.}, $\rho(M_x(z))\leq \rho(M_x(v))$,  whenever $z\sim x$.
\end{lem}
\begin{proof}
Let $x$  be a vertex in $T$ other than  $v$. Let $x\in B_v(u_1)$ and $u_1 \sim v$.  Since there are two or more  Perron branches at $v$ in $T$,  there exists a vertex $u_2$  other than $u_1$ such that $u_2 \sim v$ and  $B_v(u_2)$ is a Perron branch  at $v$ in $T$, {\it i.e.,} $\rho(M_v)=\rho(M_v(u_2))$. Thus $B_{x}(u_2)=B_x(v)$ and $B_v(u_2)\subset B_x(v)$.  Next, for any $z\sim x$ with $z\notin B_x(v)$, we have $B_x(z)\subset B_v(z)$. By Lemma~\ref{lem:big-branch-sr}, we get 
$$\rho(M_v)=\rho(M_v(u_2))\leq \rho(M_x(v)) \mbox{ and } \rho(M_x(z))\leq \rho(M_v(z))\leq \rho(M_v).$$ 
This implies that  $\rho(M_x(z))\leq \rho(M_x(v))$, thereby  completing the proof.
\end{proof}

\begin{theorem}\label{thm:ch-vertex}
Let $T$ be a tree with positive definite matrix weights on its edges. If there exists a vertex  $v$ in $T$ such that there are two or more  Perron branches at $v$, then $v$ is a unique vertex with such a property. Moreover, if  $x$ is a vertex other than $v$, then the unique Perron branch at $x$ in $T$ is the branch which contains the vertex $v$.
\end{theorem}
\begin{proof}
Let $v$ be a vertex in $T$  such that there are two or more  Perron branches at $v$. We consider the following two cases to complete the proof.\\

\noindent$\underline{\textbf{Case 1:}}$ There is  a unique Perron branch  at $u$ in $T$, whenever $u$ adjacent to $v$.\\

Let $x$  be a vertex in $T$ other than  $v$ and  $x\in B_v(u_1)$ for  $u_1 \sim v$. Then $B_{u_1}(v)\subseteq B_x(v)$. For  any $y\sim x $ with $ y \notin B_x(v)$, we have $B_x(y) \subseteq B_{u_1} (y)$. By Lemma~\ref{lem:big-branch-sr}, we get
\begin{equation}\label{eqn:case1}
\rho(M_{u_1}(v))\leq \rho(M_{x}(v)) \mbox{ and } \rho(M_{x}(y))\leq \rho(M_{u_1}(y)). 
\end{equation}

Using Lemma~\ref{lem:ch-ver-lem-pd}, we see that  $B_{u_1}(v)$  is a Perron branch  at $u_1$ in $T$. By our assumption for this case $B_{u_1}(v)$  is the  unique Perron branch  at $u_1$ in $T$. Thus $\rho(M_{u_1}(y)) < \rho(M_{u_1}(v))$. Hence, by Equation~\eqref{eqn:case1}, we have $\rho(M_{x}(y))< \rho(M_{x}(v))$.  Therefore, $B_x(v)$  is the unique Perron branch  at $x$ in $T$ and the result follows.\\

\noindent$\underline{\textbf{Case 2:}}$ There exists a  vertex $u_1$ adjacent to $v$ such that there are two or more Perron branches  at $u_1$ in $T$.\\

By the  hypothesis and our assumption for this case, there are two or more Perron branches of $T$ at both the vertices $v$ and $u_1$. Therefore,   Lemma~\ref{lem:ch-ver-lem-pd} shows that  $B_v(u_1)$ and $B_{u_1}(v)$ are Perron branches of $T$ at $v$ and $u_1$, respectively.   Thus
\begin{equation}\label{eqn:case2.1}
\rho(M_v)=\rho(M_v(u_1)) \mbox{ and } \rho(M_{u_1})=\rho(M_{u_1}(v)).
\end{equation}
Since there are two or more Perron branches  at $v$ in $T$,  there exists $w \sim v \; (w\neq u_1)$ such that $B_v(w)$ is a Perron branch  at $v$ in $T$. Similarly, there exists  $u_2 \sim u_1 \; (u_2 \neq v)$ such that  $B_{u_1}(u_2)$ is a Perron branch at  $u_1$ in $T$. Thus  $$B_v(w) \subset B_{u_1}(w)=  B_{u_1}(v) \mbox{ and } B_{u_1}(u_2) \subset B_v(u_2)=  B_v(u_1). $$
 Hence, by  Lemma~\ref{lem:big-branch-sr} and   Equation~\eqref{eqn:case2.1}, we get 
 $$\rho(M_v)=\rho(M_v(w))\leq \rho(M_{u_1}(v))\leq \rho(M_{u_1})=\rho(M_{u_1}(u_2)) \leq \rho(M_v(u_1))=\rho(M_v).$$
Therefore, 
\begin{equation}\label{eqn:case2.2}
\rho(M_v)=\rho(M_v(u_1))=\rho(M_{u_1}(u_2))=\rho(M_{u_1}).
\end{equation}

Next, we consider all the branches of $T$ at $u_2$ except for $B_{u_2}(u_1)$ and choose a branch  such that the bottleneck matrix is with maximum spectral radius   in the following way. Let $u_3 \sim u_2$ and  $B_{u_2}(u_3)$ be a branch  such that $M_{u_2}(u_3)$ is with maximum spectral radius amongst all the branches of $T$ at $u_2$ except for $B_{u_2}(u_1)$ and  repeat the process until we  reach a pendant vertex. Thus, there exists  a path $v=u_0\sim u_1 \sim \cdots \sim u_r$ such that $u_r$ is a pendant vertex, and   $B_{u_i}(u_{i+1})$ is a branch  such that $M_{u_i}(u_{i+1})$ is with maximum spectral radius amongst all the branches of $T$ at $u_i$ except for the branch $B_{u_i} (u_{i-1})$ for $1\leq i\leq r-1$.

For $1\leq i\leq r-1$, let $\widehat{M}_{u_i}$  denote the principal submatrix of  $M_{u_i}$ obtained by deleting the block $M_{u_i}(u_{i-1})$ from $M_{u_i}$, {\it i.e.,}
$$M_{u_i}=\left[
\begin{array}{c|c}
\widehat{M}_{u_i} &  \mathbf{0} \\
\midrule
\mathbf{0}  &  M_{u_i}(u_{i-1})
\end{array}
\right],$$
 and let $e_i$ denote the edge between the vertices $u_{i-1}$ and $u_{i}$.  Then 
\begin{equation}\label{eqn:case2.3}
\begin{cases}
M_{u_{i-1}}(u_i)= \widehat{\widehat{M\, }}_{u_i}+ J\otimes [W(e_i)^{-1}] & \mbox{ for }1\leq i\leq r-1, \\
M_{u_{r-1}}(u_r)= W(e_r)^{-1},
\end{cases}
\end{equation}
where  $\widehat{\widehat{M\, }}_{u_i}=\left[
\begin{array}{c|c}
\widehat{M}_{u_i} &  \mathbf{0} \\
\midrule
\mathbf{0}  &  \mathbf{0}_{s\times s}
\end{array}
\right]$  if the matrix weights on the edges are of order $s\times s$.

 By  our construction $\widehat{\widehat{M\, }}_{u_i}$ is a block diagonal matrix and $M_{u_i}(u_{i+1})$ is of maximum spectral radius amongst all the blocks of  $\widehat{\widehat{M\, }}_{u_i}$. Hence
\begin{equation}\label{eqn:bl-sp}
\rho(\widehat{\widehat{M\, }}_{u_i})=\rho(\widehat{M}_{u_i})= \rho(M_{u_i}(u_{i+1}))  \mbox{ for } 1\leq i\leq r-1.
\end{equation}
Thus, if $\mathbf{x}_{i+1}$ is an eigenvector of $M_{u_i}(u_{i+1})$ corresponding to 
$\rho(M_{u_i}(u_{i+1}))$, then  the vector $\widehat{\mathbf{x}}_{i+1}=(\mathbf{x}_{i+1}, \mathbf{0}, \ldots, \mathbf{0})$ of conformal order, is an eigenvector of $\widehat{\widehat{M\, }}_{u_i}$ corresponding to $\rho(\widehat{\widehat{M\, }}_{u_i})$.

For $i=r-1$, $\rho(\widehat{\widehat{M\, }}_{u_{r-1}})= \rho(M_{u_{r-1}}(u_r))= \rho(W(e_r)^{-1})$. Let $\mathbf{x}_r$ be an eigenvector of $W(e_r)^{-1}$  corresponding to $\rho(W(e_r)^{-1})$. Using $ \mathbf{x}_{r}^T [W(e_{r-1})^{-1}]\mathbf{x}_{r} >0 $,  Equations~\eqref{eqn:case2.3} and \eqref{eqn:bl-sp},  we have 
\begin{align*}
\widehat{\mathbf{x}}_{r}^T M_{u_{r-2}}(u_{r-1}) \widehat{\mathbf{x}}_{r}&= \widehat{\mathbf{x}}_{r}^T \widehat{\widehat{M\, }}_{u_{r-1}}\widehat{\mathbf{x}}_{r} + \widehat{\mathbf{x}}_{r}^T( J\otimes [W(e_{r-1})^{-1}]) \widehat{\mathbf{x}}_{r} \\
&= \mathbf{x}_{r}^T [W(e_r)^{-1}]\mathbf{x}_{r} 
 + \mathbf{x}_{r}^T [W(e_{r-1})^{-1}]\mathbf{x}_{r}\\
&> \rho(W(e_r)^{-1}) \\ 
&= \rho(M_{u_{r-1}}(u_r)),
\end{align*}
which implies that $ \rho(M_{u_{r-2}}(u_{r-1})) > \rho(M_{u_{r-1}}(u_r)).$ 

 Further, suppose $\mathbf{x}_{r-1} \in$ Null($J\otimes [W(e_{r-1})^{-1}]$). By Equation~\eqref{eqn:case2.3}, we have
 \begin{align*}
 \mathbf{x}_{r-1}^T M_{u_{r-2}}(u_{r-1}) \mathbf{x}_{r-1} &= \mathbf{x}_{r-1}^T \widehat{\widehat{M\, }}_{u_{r-1}}\mathbf{x}_{r-1} + \mathbf{x}_{r-1}^T( J\otimes [W(e_{r-1})^{-1}]) \mathbf{x}_{r-1}\\
 &= \mathbf{x}_{r-1}^T \widehat{\widehat{M\, }}_{u_{r-1}}\mathbf{x}_{r-1}.
 \end{align*}
The min-max theorem yields that $\rho(M_{u_{r-2}}(u_{r-1}))\leq \rho(M_{u_{r-1}}(u_r))$, which is a contradiction.  Thus, $\mathbf{x}_{r-1} \notin$ Null($J\otimes [W(e_{r-1})^{-1}]$)  and  hence     Remark~\ref{rem:rem1} now implies  $\widehat{\mathbf{x}}_{r-1} \notin$ Null( $J\otimes [W(e_{r-2})^{-1}]$), where $\widehat{\mathbf{x}}_{r-1}=(\mathbf{x}_{r-1}, \mathbf{0}, \ldots, \mathbf{0})$  is an eigenvector of $\widehat{\widehat{M\, }}_{u_{r-2}}$ corresponding to  
$\rho(\widehat{\widehat{M\, }}_{u_{r-2}})= \rho(M_{u_{r-2}}(u_{r-1}))$. Using $\widehat{\mathbf{x}}_{r-1}^T( J\otimes [W(e_{r-2})^{-1}]) \widehat{\mathbf{x}}_{r-1} >0$,  Equations~\eqref{eqn:case2.3} and \eqref{eqn:bl-sp}, we have
\begin{align*}
\widehat{\mathbf{x}}_{r-1}^T M_{u_{r-3}}(u_{r-2}) \widehat{\mathbf{x}}_{r-1}&= \widehat{\mathbf{x}}_{r-1}^T \widehat{\widehat{M\, }}_{u_{r-2}} \widehat{\mathbf{x}}_{r-1} + \widehat{\mathbf{x}}_{r-1}^T( J\otimes [W(e_{r-2})^{-1}]) \widehat{\mathbf{x}}_{r-1} \\
&> \rho(M_{u_{r-2}}(u_{r-1})),
\end{align*} 
which implies that $ \rho(M_{u_{r-3}}(u_{r-2})) > \rho(M_{u_{r-2}}(u_{r-1})).$ Proceeding inductively  we have
$$  \rho(M_{u_{i-1}}(u_{i})) >  \rho(M_{u_{i}}(u_{i+1}))  \mbox{ for } 1\leq i\leq r-1,$$
which is a contradiction to  Equation~\eqref{eqn:case2.2} as $v=u_0$.  

Therefore, the assumption for  Case $2$ is not valid,  which implies that for any adjacent vertex $u$ of $v$, there is a unique Perron branch at $u$ in $T$. Hence, combining the conclusions of  Case $1$ and Lemma~\ref{lem:ch-ver-lem-pd}, the desired result follows. 
\end{proof}

In view of Lemma~\ref{lem:big-branch-sr} and the results in  Theorems~\ref{thm:ch-edge1} and~\ref{thm:ch-vertex}, it is easy to see that Results~\ref{result:1} and~\ref{result:2} hold if the edge weights of the tree $T$ are positive definite matrices. 


\subsection{Results for Lower (or Upper) Triangular Matrix Weights}\label{sec:ltw}
In this section, we consider trees where the weights on the edges are lower (or upper) triangular matrices with positive diagonal entries.  Since the arguments in the proofs for lower triangular matrix weights and upper triangular matrix weights are similar,  we only provide results for lower triangular matrix weights with positive diagonal entries. We begin with a few preliminary results.

Let $\mathcal{M}_n(\mathbb{R})$ be the class of real matrices of order $n \times n$ and  $A,B \in \mathcal{M}_n(\mathbb{R})$.   Matrices $A$ and $B$ are said to be permutation equivalent, denoted by $A \simeq B$, if there exists a permutation matrix $\mathbf{P}$ such that $A =\mathbf{P} B \mathbf{P}^T.$ We now state a result related to  permutation equivalence and the Kronecker product.
\begin{prop}\label{prop:kroneck-similar}\cite{Henderson}
Let $A\in \mathcal{M}_m(\mathbb{R})$ and $B \in \mathcal{M}_n(\mathbb{R})$. Let $\mathbf{P}$ be a permutation matrix of order $mn \times mn$ such that
$$P=\sum_{i=1}^m (\mathbf{e}_i \otimes I_n \otimes \mathbf{e}_i^T),$$
where $\{\mathbf{e_i} : 1\leq i \leq m\}$ is the standard basis of $\mathbb{R}^m$. Then  $A\otimes B=\mathbf{P} (B \otimes A) \mathbf{P}^T,$ {\it i.e.,} $A\otimes B$ and $B \otimes A$ are permutation equivalent.
\end{prop}

The  permutation matrix $P=\sum_{i=1}^m (\mathbf{e}_i \otimes I_n \otimes \mathbf{e}_i^T)$ is also called the vec-permutation matrix and is denoted by $I_{m,n}$ (for details, see~\cite{Henderson}).  We now prove a lemma that plays an important role in proving Results~\ref{result:1} and~\ref{result:2}. To prove this result, we use the fact that the vec-permutation matrix $P=I_{m,n}$ does not depend on the entries of the matrices $A$ and $B$,  but depends only on the order of these matrices.

\begin{lem}\label{lem:perm-equv}
 For $1\leq i,j\leq m$, let $X_{ij}$ be  matrices of order $s\times s$. For $1\leq l,k\leq s$, let $\widetilde{X}_{lk}$ be  matrices of order $m\times m$ such that $(\widetilde{X}_{lk})_{ij}= (X_{ij})_{lk}$, {\it i.e., } the $(i,j)^{th}$ entry of $\widetilde{X}_{lk}$ is   the $(l,k)^{th}$ entry of  $X_{ij}$. If $X=[X_{ij}]$ and $\widetilde{X}=[\widetilde{X}_{lk}]$ are the block matrices of order $ms \times ms$, then $X$ and $\widetilde{X}$ are permutation equivalent.
\end{lem}
\begin{proof}
Let $\{\mathbf{e}_i : 1\leq i \leq m\}$ be the standard basis of $\mathbb{R}^m$ and $\{\mathbf{f}_l : 1\leq l \leq s\}$ be the standard basis of $\mathbb{R}^s$. For $1\leq i,j\leq m$ and $1\leq l,k\leq s$, let 
$$E_{ij}= \mathbf{e}_i {\mathbf{e}_j}^T  \mbox{ and } F_{lk}=\mathbf{f}_l{\mathbf{f}_k}^T. $$
Then, $\{E_{ij} : 1\leq i,j\leq m\}$ is the standard basis of $\mathcal{M}_{m}(\mathbb{R})$ and $\{F_{lk} : 1\leq l,k\leq s\}$ is the standard basis of $\mathcal{M}_{s}(\mathbb{R})$. Moreover, $$\{E_{ij} \otimes  F_{lk} : 1\leq i,j\leq m, 1\leq l,k\leq s \} \mbox{ and } \{F_{lk} \otimes E_{ij} : 1\leq i,j\leq m, 1\leq l,k\leq s \}$$  are both bases of
 $\mathcal{M}_{ms}(\mathbb{R})$. Thus,
\begin{equation}\label{eqn:X}
X = \sum_{i,j} X_{ij} \otimes E_{ij}\\
    = \sum_{ i,j } \left(  \sum_{ l,k } (X_{ij})_{lk}\  F_{lk} \right) \otimes E_{ij} 
   = \sum_{i,j }  \sum_{ l,k }   (X_{ij})_{lk}\ (F_{lk} \otimes E_{ij}),
\end{equation}
and 
\begin{equation}\label{eqn:X-tilde}
\widetilde{X} = \sum_{l,k} \widetilde{X}_{lk} \otimes F_{lk}\\
    = \sum_{l,k} \left( \sum_{ i,j } (\widetilde{X}_{lk})_{ij}\ E_{ij}  \right) \otimes F_{lk} 
  = \sum_{l,k} \sum_{ i,j } (\widetilde{X}_{lk})_{ij} \ (E_{ij}  \ \otimes F_{lk} ).
\end{equation}
 By Proposition~\ref{prop:kroneck-similar}, there exists a permutation matrix $\mathbf{P}$ such that 
 \begin{equation}\label{eqn: X-eqv-X-tilde}
 F_{lk} \otimes E_{ij}= \mathbf{P}(E_{ij} \otimes  F_{lk})\mathbf{P}^T \mbox{ for all } 1\leq i,j\leq m \mbox{ and } 1\leq l,k\leq s.
\end{equation}
Using  $(\widetilde{X}_{lk})_{ij}= (X_{ij})_{lk}$ and  Equations~\eqref{eqn:X} -~\eqref{eqn: X-eqv-X-tilde}, we have 
$$X=\sum_{i,j}  \sum_{l,k }   (\widetilde{X}_{lk})_{ij}\  \mathbf{P}(E_{ij} \otimes  F_{lk})\mathbf{P}^T = \mathbf{P} \left(\sum_{ i,j }  \sum_{l,k }   (\widetilde{X}_{lk})_{ij}\  (E_{ij} \otimes  F_{lk})\right)\mathbf{P}^T =\mathbf{P} \widetilde{X} \mathbf{P}^T.
 $$
 This completes the proof.
\end{proof}

\begin{rem}\label{obs:similar}
\begin{enumerate}
\item[$1.$]  Under the hypothesis of Lemma~\ref{lem:perm-equv},  if $X_{ij}$'s are lower triangular  matrices, then $\widetilde{X}$ is a lower triangular block matrix, {\it i.e.}, 
 \begin{equation*}
{\small
\widetilde{X}=\left[
\begin{array}{c |c| c| c}
   \widetilde{X}_{11} & \mathbf{0}& \dots & \mathbf{0} \\
    \midrule
    \widetilde{X}_{21} & \widetilde{X}_{22} & \dots & \mathbf{0} \\
    \midrule    
    \vdots & \vdots & \ddots & \vdots \\
    \midrule
    \widetilde{X}_{s1}  & \widetilde{X}_{s2}& \dots & \widetilde{X}_{ss}
\end{array}
\right].}
\end{equation*}
Since $X$ and $\widetilde{X}$ are permutation equivalent, 
 $\sigma(X)=  \sigma(\widetilde{X})=  \bigcup_{i=1}^s \sigma(\widetilde{X}_{ii}).$ It is easy to see that a similar assertion can be made  whenever $X_{ij}$'s are upper triangular  matrices.

\item[$2.$] If $W=[W_{ij}]$ is an invertible lower (or upper) triangular matrix, then $W^{-1}$ is a lower (or upper) triangular matrix and the diagonal entries of  $W^{-1}$ are given by $(W^{-1})_{ii}= \dfrac{1}{W_{ii}}$.  
\end{enumerate}  
\end{rem}

A graph $G=(V, E)$ with weights assigned to its edges can also be presented as an ordered pair $(G, \{W(e)\}_{e\in E})$, where $G$ is the underlying graph and $\{W(e)\}_{e\in E}$ is the set of weights assigned to the edges in $E$. Using this representation, we define trees with positive edge weights obtained from a tree whenever  the weights assigned to its edges are lower triangular matrices with positive diagonal entries.

\begin{defn}
Let $T=(V,E)$ be  a tree such that the  weights on the edges of $T$ are $s\times s$ lower  triangular matrices with positive diagonal entries.  Let $W(e)= [w_{ij}(e)]$ denote the ($s\times s$ lower triangular  matrix) weight on the edge $e\in E$ such that $w_{jj}(e) > 0$ for all $1\leq j\leq s$.  For $1\leq j\leq s$,  let $$T^{(j)}=(T, \{w_{jj}(e)\}_{e\in E})$$ be  the tree $T=(V,E)$ with positive edge weights $\{w_{jj}(e)\}_{e\in E}$. We say that $ T^{(j)}$ is a tree with positive edge weights induced by $T=(V,E)$ with  $s\times s$ lower triangular matrix edge weights.
\end{defn}  

\begin{ex}
\begin{figure}[ht]
\centering
\begin{tikzpicture} [scale=1]
\Vertex[x=-4,y=1,label=$v_1$]{A}
\Vertex[x=-4,y=-1,label=$v_2$]{B}
\Vertex[x=-2,y=0,label=$v_3$]{C}
\Vertex[x=0,y=0,label=$v_4$]{D}
\Vertex[x=2,y=0,label=$v_5$]{E}
\Vertex[x=4,y=1,label=$v_6$]{F}
\Vertex[x=4,y=-1,label=$v_7$]{G}

\Edge[label=$e_1$,position=above](A)(C)
\Edge[label=$e_2$,position=above](B)(C)
\Edge[label=$e_3$,position=above](C)(D)
\Edge[label=$e_4$,position=above](D)(E)
\Edge[label=$e_5$,position=above](E)(F)
\Edge[label=$e_6$,position=above](E)(G)
\end{tikzpicture} 
\caption{}
    \label{fig:A1}
\end{figure}

Let $V=\{v_1,v_2,v_3,v_4,v_5, v_6,v_7\}$ and  $E=\{e_1,e_2,e_3,e_4,e_5, e_6\}$. Consider the tree $T=(V,E)$, as shown in  Figure~\eqref{fig:A1} with the lower triangular matrix weights 
\begin{align*}
&\mathcal{W}=\left\{\  W(e_1)=\left[
\begin{array}{r r r}
3  &  0  &  0 \\
0  &  2  &  0 \\
10 &  17 &  5
\end{array}
\right],
W(e_2)=\left[
\begin{array}{r r r}
9   &  0  &  0 \\
-3  &  4  &  0 \\
2   &  5  &  17
\end{array}
\right],
W(e_3)=\left[
\begin{array}{r r r}
1    &  0   &  0 \\
-11  &  12  &  0 \\
0    &  -4  &  4
\end{array}
\right],\right.\\
&\qquad \quad\  \left.
W(e_4)=\left[
\begin{array}{r r r}
11    &  0   &  0 \\
2  &  1  &  0 \\
-16   &  0  &  3
\end{array}
\right],
W(e_5)=\left[
\begin{array}{r r r}
15    &  0   &  0 \\
3     &  6  &  0 \\
2     &  1  &  9
\end{array}
\right],
W(e_6)=\left[
\begin{array}{r r r}
7   &  0   &  0 \\
-10    & 8  &  0 \\
-9    &  -1  &  6
\end{array}
\right]\ \ \right\}.
\end{align*}
Let 
\begin{align*}
&\mathcal{W}^{(1)}=\left\{ W(e_1)=3,  W(e_2)=9,  W(e_3)=1, W(e_4)=11,  W(e_5)=15,  W(e_6)=7\right\},\\
\\
&\mathcal{W}^{(2)}=\left\{ W(e_1)=2,  W(e_2)=4,  W(e_3)=12, W(e_4)=1,  W(e_5)=6,  W(e_6)=8\right\},\\
\\
&\mathcal{W}^{(3)}=\left\{ W(e_1)=5,  W(e_2)=17,  W(e_3)=4, W(e_4)=3,  W(e_5)=9,  W(e_6)=6\right\}.\\
\end{align*}
Let    $T^{(1)}=(T, \mathcal{W}^{(1)} )$, $T^{(2)}=(T, \mathcal{W}^{(2)})$  and $T^{(3)}=(T, \mathcal{W}^{(3)} )$. Thus, $T^{(1)}, T^{(2)}$  and $T^{(3)}$   are the trees with positive edge weights induced by $T=(V,E)$ with $3\times 3$ lower triangular matrix edge weights $\mathcal{W}$.
\end{ex}

Let $T=(V,E)$ be  a tree on $n$ vertices such that the  weights on the edges of $T$ are $s\times s$ lower  triangular matrices with positive diagonal entries.  Let $W(e)= [w_{ij}(e)]$ denote the ($s\times s$ lower triangular  matrix) weight on the edge $e\in E$ such that $w_{jj}(e) > 0$ for all $1\leq j\leq s$. For $1\leq j\leq s$, let $ T^{(j)}$   be the trees with positive edge weights induced by $T=(V,E)$ with $s\times s$ lower triangular matrix edge weights. However, $(T,\{W(e)\}_{e\in E})$  is the tree $T=(V,E)$ with the matrix weights $\{W(e)\}_{e\in E}$ on its edges, simply written as $T$.  For $1\leq j\leq s$,  let $L(T)$ and $L(T^{(j)})$ be the Laplacian matrices of $T$ and $T^{(j)}$, respectively. Then, $L(T)$ is a matrix of order $ns\times ns$ and $L(T^{(j)})$ is a matrix  of order $n \times n$ for $1\leq j\leq s$.  Thus,   using  Remark~\ref{obs:similar} $(1)$ for the Laplacian matrix $L(T)$ of $T$, we have
 \begin{equation}\label{eqn:lap-ltw}
{\small
L(T)\simeq \widetilde{L(T)}=\left[
\begin{array}{c |c| c| c}
    L(T^{(1)}) & \mathbf{0}& \dots & \mathbf{0} \\
    \midrule
    \ast & L(T^{(2)}) & \dots & \mathbf{0} \\
    \midrule    
    \vdots & \vdots & \ddots & \vdots \\
    \midrule
    \ast  & \ast & \dots &L(T^{(s)})
\end{array}
\right]} \mbox{ and } \sigma(L(T))= \bigcup_{j=1}^s \sigma(L(T^{(j)}))
\end{equation}
and hence  the eigenvalues of $L(T)$ are nonnegative. 

\begin{lem}\label{lem:sp-bl}
Let $T=(V,E)$ be  a tree  such that the  weights on the edges of $T$ are $s\times s$ lower  triangular matrices with positive diagonal entries.  For $1\leq j\leq s$, let $ T^{(j)}$   be the trees with positive edge weights induced by $T=(V,E)$ with $s\times s$ lower triangular matrix edge weights. Let $v\in V$ and     $B$ be a branch  at $v$.  For $1\leq j\leq s$, let  $ M_v(B)$  and $ M_v^{(j)}(B)$ be the bottleneck matrix of  the branch $B$ at $v$ in  $T$ and $T^{(j)}$, respectively. Then, 
$$\sigma(M_v(B))= \bigcup_{j=1}^{s}\sigma( M_v^{(j)}(B)) \mbox{ and } \rho(M_v(B))= \max_{1\leq j \leq s} \rho( M_v^{(j)}(B)).$$
\end{lem} 
\begin{proof}
Let $T=(V,E)$ be  a tree and $W(e)= [w_{ij}(e)]$ denote the ($s\times s$ lower triangular  matrix) weight on the edge $e\in E$ such that $w_{jj}(e) > 0$ for all $1\leq j\leq s$.  Let $v\in V$ and $B$ be a branch  at $v$ consisting of $k$ vertices. By definition $T^{(j)}=(T,[w_{jj}(e)]_{e\in E})$ for $1\leq j\leq s$. Let  $ M_v(B)$  and $ M_v^{(j)}(B)$ be the bottleneck matrix of  the branch $B$ at $v$ in  $T$ and $T^{(j)}$, respectively. 

Let $x,y \in B$. For $1\leq j\leq s$, the bottleneck matrix $ M_v^{(j)}(B)$ is a $k \times k$ matrix,  and the entry at the $(x,y)^{th}$ position  of $M_v^{(j)}(B)$ is given by
\begin{equation}\label{eqn:bt-j}
\ds \sum_{e \in \mathcal{P}(x,v)\, \cap \, \mathcal{P}(y,v)} \frac{1}{w_{jj}(e)}.
\end{equation}
The bottleneck matrix $ M_v(B)$ is a $ks \times ks$ matrix  such that ($M_v(B)$ as a block  matrix)  the block at the $(x,y)^{th}$ position   of $M_v(B)$ is an $s\times s$ matrix $W_{xy}$ (say), and   is given by 
$$W_{xy}=\ds \sum_{e \in \mathcal{P}(x,v)\, \cap \, \mathcal{P}(y,v)} W(e)^{-1}.$$ 
Since the edge weights $\{W(e)\}_{e \in E}$ are lower triangular matrices, by  Remark~\ref{obs:similar} $(2)$, we see that $W_{xy}$ is an $s \times s$  lower triangular matrix and its  diagonal entries are given by 
$$(W_{xy})_{jj}=\ds \sum_{e \in \mathcal{P}(x,v)\, \cap \, \mathcal{P}(y,v)} \frac{1}{w_{jj}(e)} \mbox{ for } 1\leq j\leq s. $$
Thus, using Remark~\ref{obs:similar} $(1)$ for $M_v(B)$ and Equation~\eqref{eqn:bt-j}, we have 
 \begin{equation}\label{eqn:bottle-lt}
{\small
M_v(B)\simeq \widetilde{M_v(B)}=\left[
\begin{array}{c |c| c| c}
    M_v^{(1)}(B) & \mathbf{0}& \dots & \mathbf{0} \\
    \midrule
    \ast & M_v^{(2)}(B) & \dots & \mathbf{0} \\
    \midrule    
    \vdots & \vdots & \ddots & \vdots \\
    \midrule
    \ast  & \ast & \dots & M_v^{(s)}(B)
\end{array}
\right].}
\end{equation}
This completes the proof.
\end{proof}

\begin{theorem}
Let $T=(V,E)$ be  a tree  such that the  weights on the edges of $T$ are $s\times s$ lower  triangular matrices with positive diagonal entries.  For $1\leq j\leq s$, let $ T^{(j)}$   be the trees with positive edge weights induced by $T=(V,E)$ with $s\times s$ lower triangular matrix edge weights. 
\begin{enumerate}
\item[(a)] Let $v\in V$ with $\deg(v)=r$ and    $B_i$ be the branches  at $v$ for $1\leq i\leq r$. For $1\leq j\leq s$, let  $ M_v(B_i)$  and $ M_v^{(j)}(B_i)$ be the bottleneck matrix of  the branch $B_i$ at $v$ in  $T$ and $T^{(j)}$, respectively. Then,
$$\rho(M_v)=\ds   \max_{1\leq i \leq r \atop 1\leq j \leq s} \rho( M_v^{(j)}(B_i)) =  \max_{ 1\leq j \leq s} \rho( M_v^{(j)}).$$

\item[(b)]  The Perron value and the Perron branch are well-defined.  Moreover,   if $B$ is a (unique) Perron branch  at a  vertex $v$ in $T^{(j)}$ for all $1\leq j \leq s$, then $B$ is a  (unique)  Perron branch  at $v$ in $T$ and  
$$\rho(M_v)=\rho(M_v(B))= \max_{ 1\leq j \leq s} \rho( M_v^{(j)}(B)).$$ 
\end{enumerate}
\end{theorem}
\begin{proof}
Let $v\in V$ with $\deg(v)=r$ and    $B_i$ be the branches  at $v$ for $1\leq i\leq r$. For $1\leq j\leq s$, let  $ M_v(B_i)$  and $ M_v^{(j)}(B_i)$ be the bottleneck matrix of  the branch $B_i$ at $v$ in  $T$ and $T^{(j)}$, respectively. Thus, for $1\leq i\leq r$,  using Equation~\eqref{eqn:bottle-lt}, we have 
$$
{\small
M_v=\left[
\begin{array}{c |c| c| c}
    M_v(B_1) & \mathbf{0}& \dots & \mathbf{0} \\
    \midrule
    \mathbf{0} & M_v(B_2) & \dots & \mathbf{0} \\
    \midrule    
    \vdots & \vdots & \ddots & \vdots \\
    \midrule
    \mathbf{0} & \mathbf{0}& \dots & M_v(B_r)
\end{array}
\right]} 
 \mbox{ and }
{\small
M_v(B_i)\simeq \left[
\begin{array}{c |c| c| c}
    M_v^{(1)}(B_i) & \mathbf{0}& \dots & \mathbf{0} \\
    \midrule
    \ast & M_v^{(2)}(B_i) & \dots & \mathbf{0} \\
    \midrule    
    \vdots & \vdots & \ddots & \vdots \\
    \midrule
    \ast  & \ast & \dots & M_v^{(s)}(B_i)
\end{array}
\right].}$$
Then, by Lemma~\ref{lem:sp-bl}, we have
\begin{equation}\label{eqn:perr}
\sigma(M_v(B_i))= \bigcup_{j=1}^{s}\sigma( M_v^{(j)}(B_i)) \mbox{ and } \rho(M_v(B_i))= \max_{1\leq j \leq s} \rho( M_v^{(j)}(B_i)),
\end{equation}
 and hence the spectral radius of $M_v$ is given by 
\begin{equation}\label{eqn:sp-rd-lt}
\rho(M_v)=\ds  \max_{1\leq i\leq r} \rho( M_v(B_i))= \max_{1\leq i \leq r \atop 1\leq j \leq s} \rho( M_v^{(j)}(B_i)) =  \max_{ 1\leq j \leq s} \rho( M_v^{(j)}).
\end{equation}
This proves part $(a).$

Next, for  $1\leq i\leq r$ and $1\leq j\leq s$, the eigenvalues of $M_v^{(j)}(B_i)$ are positive  and hence  $\rho(M_v^{(j)}(B_i))$ is necessarily an eigenvalue of $M_v^{(j)}(B_i)$. Therefore, by Equation~\eqref{eqn:perr}, the definition of Perron value and Perron branch in Section~\ref{sec:Lap-n-bot} are well-defined for trees  where weights on the edges are lower triangular matrices with positive diagonal entries. Hence,  part $(b)$ follows from  Equation~\eqref{eqn:sp-rd-lt}.
\end{proof}

We first prove  results for trees   where the edge weights are matrices of order $2\times 2$. Let $T=(V,E)$ be  a tree such that the  weights on the  edges of $T$ are    $2\times 2$ lower  triangular matrices   with positive diagonal entries.  Let
$T^{(1)}$ and $T^{(2)}$  be the  trees with positive edge weights induced by $T=(V,E)$ with $2\times 2$ lower triangular matrix edge weights.   Let $\mathcal{C}_{T^{(1)}}$ and  $\mathcal{C}_{T^{(2)}}$ denote the set of characteristic vertices of  $T^{(1)}$ and $T^{(2)}$, respectively. The strategy adopted to achieve our goal is as follows:  We consider all the possible cases for $\mathcal{C}_{T^{(1)}}$ and  $\mathcal{C}_{T^{(2)}}$, and for each of these cases  we  use Propositions~\ref{prop:ch-edge}~-~\ref{prop:perron-branch} for  trees $T^{(1)}$ and $T^{(2)}$ to show that  Results~\ref{result:1} and~\ref{result:2} hold  true. We begin   by considering  $\mathcal{C}_{T^{(1)}} \cap \mathcal{C}_{T^{(2)}} \neq \emptyset$ and $  \mathcal{C}_{T^{(1)}} \cap \mathcal{C}_{T^{(2)}} = \emptyset$ as separate cases.

\begin{lem}\label{lem:s=2-1}
Let $T=(V,E)$ be  a tree such that the  weights on the  edges of $T$ are    $2\times 2$ lower  triangular matrices   with positive diagonal entries.  Let $T^{(1)}$ and $T^{(2)}$  be the  trees with positive edge weights induced by $T=(V,E)$ with $2\times 2$ lower triangular matrix edge weights such that $\mathcal{C}_{T^{(1)}} \cap \mathcal{C}_{T^{(2)}} \neq \emptyset$. Then, one of the following cases occurs: 
\begin{enumerate}
\item[$1.$] There is a unique vertex $v$ such that there are two or more   Perron branches at $v$ in $T$. Moreover, if  $x$ is a vertex other than $v$, then the unique Perron branch at $x$ in $T$ is the branch which contains the vertex $v$.

\item[$2.$] There is  a unique  pair of vertices $u$ and $v$ with $u \sim v$  such that   the Perron branch  at $u$ in $T$ is  the branch containing $v$, while the Perron branch  at $v$  in $T$  is  the branch containing $u$.  Moreover, the unique Perron branch at any vertex $x$ in $T$ is the branch which contains at least one of the vertices $u$ or $v$.
\end{enumerate}
\end{lem}
\begin{proof}
We  consider different choices of $\mathcal{C}_{T^{(1)}}$ and  $\mathcal{C}_{T^{(2)}}$  with $\mathcal{C}_{T^{(1)}} \cap \mathcal{C}_{T^{(2)}} \neq \emptyset$, and prove that the result is true  in each of these cases. \\

$\underline{\textbf{Case 1:}} \ \;$Let  $ \mathcal{C}_{T^{(1)}} = \mathcal{C}_{T^{(2)}}=\{v\}$, {\it i.e.}, the vertex $v$ is the characteristic vertex of   $ T^{(1)} \mbox{ and  } T^{(2)}$. There exist  branches $B_{i_1}$ and $B_{i_2}$ at $v$ such that $\rho(M_{v}^{(1)})= \rho(M_{v}^{(1)}(B_{i_1}))= \rho(M_{v}^{(1)}(B_{i_2})).$ Similarly, there exist  branches $B_{j_1}$ and $B_{j_2}$ at $v$  such that $\rho(M_{v}^{(2)})= \rho(M_{v}^{(2)}(B_{j_1}))= \rho(M_{v}^{(2)}(B_{j_2})).$  If $ \rho(M_{v}^{(1)}) \geq \rho(M_{v}^{(2)})$, then by Equation~\eqref{eqn:perr} the branches $B_{i_1} \mbox{ and } B_{i_2}$ are Perron branches at $v$  for  $T$ and $\rho(M_{v})= \rho(M_{v}(B_{i_1}))= \rho(M_{v}(B_{i_2})).$ Similarly, if $ \rho(M_{v}^{(1)}) \leq \rho(M_{v}^{(2)})$, then $B_{j_1} \mbox{ and } B_{j_2}$ are Perron branches at $v$  for  $T$  and $\rho(M_{v})= \rho(M_{v}(B_{j_1}))= \rho(M_{v}(B_{j_2})).$  Further, if $x\neq v$, then $B_x(v)$ is the  unique Perron branch at $x$ in $ T^{(1)} \mbox{ and  } T^{(2)}$.  Thus, $B_x(v)$ is the  unique Perron branch  at $x$ in $T$. 

Therefore, $v$ is the unique vertex such that there are two or more   Perron branches  at $v$ and $B_x(v)$ is the unique  Perron branch  at $x$ in $T$, whenever $x\neq v$.\\

$\underline{\textbf{Case 2:}} \ \;$Let  $ \mathcal{C}_{T^{(1)}} = \mathcal{C}_{T^{(2)}}=\{u,v\}$, {\it i.e.}, the edge between the  vertices $u$ and $v$ is the characteristic edge of  $ T^{(1)} \mbox{ and  } T^{(2)}$. Then,  $B_u(v)$ is the unique  Perron branch  at $u$ in $T^{(j)}$  and $B_v(u)$ is the unique  Perron branch  at $v$ in $T^{(j)}$ for    $j=1,2$. Thus, 
$$\rho(M_u)=\rho(M_u(v))= \max_{j=1,2} \rho(M_{u}^{(j)}(v)) \mbox{ and } \rho(M_v)=\rho(M_v(u))= \max_{j=1,2} \rho(M_{v}^{(j)}(u)).$$
Hence,  $B_u(v)$ is the unique  Perron branch  at $u$ in $T$, while   $B_v(u)$ is the unique  Perron branch  at $v$ in $T$. 

Let  $x\in V$ such that $x\neq u$ and $x\neq v$.  If  $B$ is the branch at $x$ containing $u$ and $v$, then $B$ is the unique Perron branch   at $x$ in $T^{(j)}$ for $j=1,2$. Hence $B$ is the unique Perron branch   at $x$ in $T$.\\

$\underline{\textbf{Case 3:}}$ Let $ \mathcal{C}_{T^{(1)}} = \{u,v\}$ and $\mathcal{C}_{T^{(2)}}=\{v\}$, {\it i.e.}, the edge between the vertices $u$ and $v$ is the characteristic edge of  $ T^{(1)}$ and $v$ is the characteristic vertex of  $  T^{(2)}$. Thus, $B_u(v)$ is the unique  Perron branch   at $u$ in  $T^{(j)}$ for $j=1,2$. Hence 
\begin{equation}\label{eqn:subcase 1.3}
\rho(M_u)=\rho(M_u(v))= \max_{j=1,2} \rho(M_{u}^{(j)}(v)),
\end{equation}
and $B_u(v)$ is the unique  Perron branch   at $u$ in $T$. 
Further, $B_v(u)$ is the unique  Perron branch  at $v$ in $T^{(1)}$, {\it i.e.,} $\rho(M_v^{(1)})= \rho(M_{v}^{(1)}(u))$  and there exist branches $B_{j_1}$ and $B_{j_2}$ at $v$  such that $\rho(M_{v}^{(2)})= \rho(M_{v}^{(2)}(B_{j_1}))= \rho(M_{v}^{(2)}(B_{j_2})).$ Therefore, the following cases arise:
\begin{itemize}
\item If $\rho(M_{v}^{(1)}) \leq \rho(M_{v}^{(2)})$, then $\rho(M_{v}^{(2)})=\rho(M_{v})= \rho(M_{v}(B_{j_1}))= \rho(M_{v}(B_{j_2})).$ Thus,  there are two or more   Perron branches of $T$ at $v$ and the uniqueness of the vertex $v$ follows from an argument similar to that in Case $1$.

\item If $\rho(M_{v}^{(1)}) > \rho(M_{v}^{(2)})$, then  $\rho(M_{v}^{(1)})=\rho(M_{v})= \rho(M_{v}(u))$ and hence $B_v(u)$ is the unique  Perron branch  at $v$ in $T$. By Equation~\eqref{eqn:subcase 1.3},   $B_u(v)$ is the unique  Perron branch  at $u$ in $T$.  The uniqueness of the vertices $u$ and $v$ follows from an argument similar to that in Case $2$.
\end{itemize}

$\underline{\textbf{Case 4:}}$ Let $ \mathcal{C}_{T^{(1)}} = \{u,v\}$ and $\mathcal{C}_{T^{(2)}}=\{v,w\}$, {\it i.e.}, the edge between the vertices $u$ and $v$ is the characteristic edge of  $ T^{(1)}$ and the edge between the  vertices $v$ and $w$ is the characteristic edge of $  T^{(2)}$. 
Observe that, $B_u(v)$ is the unique  Perron branch at $u$ in $T^{(1)}$ and $T^{(2)}$. Similarly, $B_w(v)$ is the unique  Perron branch at $w$ in $T^{(1)}$ and $T^{(2)}$. Hence, $B_u(v)$ is the unique  Perron branch at $u$ in $T$, while $B_w(v)$ is the unique  Perron branch at $w$ in $T$ and
\begin{equation}\label{eqn:subcase 1.4-1}
\rho(M_u)=\rho(M_u(v)) \mbox{ and } \rho(M_w)=\rho(M_w(v)).
\end{equation}
Further, $B_v(u)$ is the unique  Perron branch at $v$ in $T^{(1)}$ and $B_v(w)$ is the unique  Perron branch at $v$ in $T^{(2)}$. Hence
\begin{equation}\label{eqn:subcase 1.4-2}
\rho(M_v)= \max \{\rho(M_v(u)),\rho(M_v(w)) \}.
\end{equation}
Therefore, the following cases arise:
\begin{itemize}
\item   If $\rho(M_v(u))=\rho(M_v(w))$, then by Equation~\eqref{eqn:subcase 1.4-2}, $\rho(M_v) =\rho(M_v(u))=\rho(M_v(w))$. Thus,  there are two or more   Perron branches  at $v$ in $T$ and the uniqueness of the vertex $v$ follows from an argument similar to that in Case $1$.

\item If $\rho(M_v(u))>\rho(M_v(w))$, then  Equations~\eqref{eqn:subcase 1.4-1} and~\eqref{eqn:subcase 1.4-2} yield that $B_u(v)$ is the unique  Perron branch  at $u$ in $T$ and $B_v(u)$ is the unique  Perron branch  at $v$ in $T$. The uniqueness of the vertices $u$ and $v$ follows from an argument similar to that in Case $2$.

\item If $\rho(M_v(u))<\rho(M_v(w))$, then  Equations~\eqref{eqn:subcase 1.4-1} and~\eqref{eqn:subcase 1.4-2} yield that $B_v(w)$ is the unique  Perron branch  at $v$ in $T$  and $B_w(v)$ is the unique  Perron branch  at $w$ in $T$. The uniqueness of the vertices $v$ and $w$ follows from an argument similar to that in Case $2$.
\end{itemize}
This completes the proof.
\end{proof}

\begin{lem}\label{lem:s=2-2}
Let $T=(V,E)$ be  a tree such that the  weights on the edges of $T$ are    $2\times 2$ lower  triangular matrices   with positive diagonal entries. Let $T^{(1)}$ and $T^{(2)}$  be the  trees with positive edge weights induced by $T=(V,E)$ with $2\times 2$ lower triangular matrix edge weights such that $\mathcal{C}_{T^{(1)}} \cap \mathcal{C}_{T^{(2)}} =\emptyset$. Then one of the following cases occurs: 
\begin{enumerate}
\item[$1.$] There is a unique vertex $v$ such that there are two or more   Perron branches at $v$ in $T$. Moreover, if  $x$ is a vertex other than $v$, then the unique Perron branch at $x$ in $T$ is the branch which contains the vertex $v$.

\item[$2.$] There is  a unique  pair of vertices $u$ and $v$ with $u \sim v$  such that   the Perron branch  at $u$ in $T$ is  the branch containing $v$, while the Perron branch  at $v$  in $T$  is  the branch containing $u$.  Moreover, the unique Perron branch at any vertex $x$ in $T$ is the branch which contains at least one of the vertices $u$ or $v$.
\end{enumerate}
\end{lem}
\begin{proof}

Let $  \mathcal{C}_{T^{(1)}}=\{v\}$ and   $\mathcal{C}_{T^{(2)}} = \{x,y\}$, where $v\neq x$ and $v\neq y$.  Without loss of generality, assume  that $y \notin \mathcal{P}(v,x)$,  where $\mathcal{P}(v,x)$ is the path joining the vertices $v$ and  $x$ such that  $\mathcal{P}(v,x): v=v_1 \sim v_2 \sim \cdots \sim  v_{p-1}\sim v_p=x$. Since $v$ is the characteristic  vertex of $T^{(1)}$, there exists a vertex $u$ adjacent to $v$ with $u \neq v_1$ such that $B_v(u)$ is a Perron branch at $v$ in $T^{(1)}$ and  $\rho(M_v^{(1)})=\rho(M_v^{(1)}(u))$. Thus,   $B_{v_i}(u)$ is the unique Perron branch at $v_i$ in  $T^{(1)}$, while  $B_{v_i}(y)$ is the unique Perron branch at $v_i$ in  $T^{(2)}$ for  $1\leq i\leq p $. Hence
 \begin{equation}\label{eqn:3.11}
 \rho(M_{v_i}^{(1)})=  \rho(M_{v_i}^{(1)}(u)) \mbox{ and } \rho(M_{v_i}^{(2)})= \rho(M_{v_i}^{(2)}(y)) \mbox{ for }  1\leq i\leq p.
 \end{equation}
Next, since $B_{v_i}(u) \subsetneq B_{v_{i+1}}(u) $ and  $B_{v_i}(y) \supsetneq B_{v_{i+1}}(y)$ for  $1\leq i\leq p-1$,  using Proposition~\ref{prop:big-branch-scalar} and Equation~\eqref{eqn:3.11}, we have
\begin{equation}\label{eqn:ineq}
\begin{cases}
\vspace*{.2cm}
\rho(M_v^{(1)})= \rho(M_{v_1}^{(1)})< \rho(M_{v_2}^{(1)})<\cdots< \rho(M_{v_{p-1}}^{(1)}) < \rho(M_{v_p}^{(1)})=\rho(M_x^{(1)}), \\
\rho(M_v^{(2)})= \rho(M_{v_1}^{(2)})> \rho(M_{v_2}^{(2)})>\cdots> \rho(M_{v_{p-1}}^{(2)}) > \rho(M_{v_p}^{(2)})= \rho(M_x^{(2)}).
\end{cases}
\end{equation}
Therefore, the following cases arise.\\

$\underline{\textbf{Case 1:}} \ \;$Let $\rho(M_v^{(1)})\geq \rho(M_v^{(2)})$. Using Equation~\eqref{eqn:ineq}, we have
\begin{equation}\label{eqn:3-1-6}
\rho(M_{v_i})= \max\{\rho(M_{v_i}^{(1)}),\rho(M_{v_i}^{(2)}) \}=\rho(M_{v_i}^{(1)}) \mbox{ for }  1\leq i\leq p.
\end{equation} 
Since $v$ is the characteristic  vertex of $T^{(1)}$,  there exist branches $B_{i_1}$ and $B_{i_2}$ at $v$  such that $\rho(M_{v}^{(1)})= \rho(M_{v}^{(1)}(B_{i_1}))= \rho(M_{v}^{(1)}(B_{i_2}))$. By  Equation~\eqref{eqn:3-1-6}, we have  $\rho(M_v)= \rho(M_{v}^{(1)}) $ and hence 
 $$\rho(M_{v})= \rho(M_{v}(B_{i_1}))= \rho(M_{v}(B_{i_2})).$$  
To show the uniqueness of the vertex $v$, let us consider the branch    $B_w(v)$, where $w\neq v$. If $w\neq v_i$ for $i=2,3,\ldots, p$, then $x \in B_w(v)$. Hence $ B_w(v)= B_w(x )$. Thus, $B_w(v)$ is the unique Perron branch at $w$ in  $ T^{(1)} \mbox{ and  } T^{(2)}$, and therefore, $B_w(v)$ is the unique Perron branch at $w$ in $T$. Next, if $w= v_i$ for $i=2,3,\ldots, p$,  using Equation~\eqref{eqn:ineq} and the assumption for this case, we have 
 $$\rho(M_{v_i}^{(1)}) > \rho(M_{v}^{(1)})  \geq \rho(M_{v}^{(2)}) >  \rho(M_{v_i}^{(2)}). $$
This implies that $\rho(M_{w})= \max \{\rho(M_{w}^{(1)}),\rho(M_{w}^{(2)})\}= \rho(M_{w}^{(1)})=\rho(M_{w}^{(1)}(v))= \rho(M_{w}(v)).$ Thus, $B_w(v)$ is the unique Perron branch at $w$ in $T$.

Therefore, $v$ is the unique vertex of $T$ such that there are two or more   Perron branches  at $v$ in $T$ and for any $w \neq v$,  $B_w(v)$ is the  unique Perron branch  at $w$ in $T$. \\

$\underline{\textbf{Case 2:}} \ \;$Let  $\rho(M_x^{(2)})\geq \rho(M_x^{(1)})$.  Using Equation~\eqref{eqn:ineq}, we have
\begin{equation}\label{eqn:3-1-7}
\rho(M_{v_i})= \max\{\rho(M_{v_i}^{(1)}),\rho(M_{v_i}^{(2)}) \}=\rho(M_{v_i}^{(2)}) \mbox{ for }  1\leq i\leq p.
\end{equation}
Since  the edge between the vertices $x$ and $y$ is the characteristic edge of $  T^{(2)}$,  $B_x(y)$ is the unique Perron branch  at $x$ in  $T^{(2)}$ and hence $\rho(M_{x}^{(2)}) = \rho(M_{x}^{(2)}(y))$. By Equation~\eqref{eqn:3-1-7},  $\rho(M_x)= \rho(M_{x}^{(2)}) = \rho(M_{x}^{(2)}(y))$ which implies that  $B_x(y)$ is the unique Perron branch  at $x$ in $T$.  Since $v\in B_y(x)$,  $ B_y(x)$ is the unique Perron branch at $y$ in  $  T^{(j)}$ for   $j=1,2$ and hence $ B_y(x)$ is the Perron branch  at  $y$ in $T$.   
 
Let $w$ be a vertex  such that $w\neq x$ and $w\neq y$. If $B$ is a branch at $w$  in $T$ containing $x$ and $y$, it can be shown that $B$ is the unique Perron branch at $w$ in  $T$  by an argument similar to that in Case $1$.\\

$\underline{\textbf{Case 3:}} \ \;$ Let $\rho(M_v^{(1)}) < \rho(M_v^{(2)})$ and $\rho(M_x^{(1)})> \rho(M_x^{(2)})$.  By Equation~\eqref{eqn:ineq}, $\rho(M_{v_{i}}^{(1)})$ is increasing and  $\rho(M_{v_{i}}^{(1)})$ is decreasing with respect to $i=1,2,\ldots, p$. Then,  one  of the following cases occurs:
\begin{enumerate}
\item[(a)] There exists a  unique vertex $v_{i_0}$ for some $2\leq i_0\leq p-1$ such that $\rho(M_{v_{i_0}}^{(1)})=\rho(M_{v_{i_0}}^{(2)}).$

\item[(b)] There exists a unique pair of  vertices $v_{i_0}$ and $v_{i_0+1}$ for some $1\leq i_0\leq p-1$ such that  $\rho(M_{v_{i_0}}^{(1)})<\rho(M_{v_{i_0}}^{(2)})$ and  $\rho(M_{v_{i_0+1}}^{(1)})>\rho(M_{v_{i_0+1}}^{(2)}).$
\end{enumerate}

For case (a), let $B_{j_1} =B_{v_{i_0}}(v)$ and $B_{j_2} =B_{v_{i_0}}(x)$. Then $B_{j_1} \neq B_{j_2}$ and $B_{j_1}$ is the unique Perron branch at $v_{i_0}$ in $T^{(1)}$, while $B_{j_2}$ is the unique Perron branch at $v_{i_0}$ in $T^{(2)}$. Thus, $\rho(M_{v_{i_0}}^{(1)})=\rho(M_{v_{i_0}}^{(1)}(B_{j_1}))$ and $\rho(M_{v_{i_0}}^{(2)})=\rho(M_{v_{i_0}}^{(2)}(B_{j_2}))$ and hence by the hypothesis, we get $\rho(M_{v_{i_0}}^{(1)})=\rho(M_{v_{i_0}}^{(2)})= \rho(M_{v_{i_0}}^{(1)}(B_{j_1}))= \rho(M_{v_{i_0}}^{(2)}(B_{j_2})).$ Therefore, by Equation~\eqref{eqn:sp-rd-lt},  we have
$$\rho(M_{v_{i_0}})=\rho(M_{v_{i_0}}(B_{j_1}))=\rho(M_{v_{i_0}}(B_{j_2})).$$
 Now we  show that $B_w(v_{i_0})$ is the unique Perron branch of $T$ at $w$, whenever $w\neq v_{i_0}$. If $w\neq v_i$ for $i=1,2,\ldots, p$, then $v,x \in B_w(v_{i_0})$. Thus, $B_w(v_{i_0})$ is the unique Perron branch at $w$ in  $ T^{(1)} \mbox{ and  } T^{(2)}$, and  hence $B_w(v_{i_0})$ is the unique Perron branch at $w$ in  $T$. If $w=v_i$ for $1\leq i < i_0$, then  $\rho(M_{v_{i}}^{(1)})<\rho(M_{v_{i}}^{(2)})$.   Thus $\rho(M_{v_{i}})=\rho(M_{v_{i}}^{(2)})$. Since $B_{v_i}(v_{i_0}) \ = B_{v_i}(x)$ is the unique Perron branch  at $v_i$ in  $T^{(2)}$,   $\rho(M_{v_{i}}^{(2)})=\rho(M_{v_{i}}^{(2)}(v_{i_0}))$ and hence
 $$\rho(M_{v_{i}})=\rho(M_{v_{i}}^{(2)})=\rho(M_{v_{i}}^{(2)}(v_{i_0})) \mbox{ for }  1\leq i < i_0.$$ 
Therefore,   $B_{v_i}(v_{i_0})$ is the unique Perron branch  at $v_i$ in  $T$ for  $1\leq i < i_0$. It is easy to see that a similar assertion can be made for $w=v_i$ for $i_0 < i\leq p.$

For  case (b), using $\rho(M_{v_{i_0}}^{(1)})<\rho(M_{v_{i_0}}^{(2)})$ and that $B_{v_{i_0}}(v_{i_0 +1})= B_{v_{i_0}}(x)$ is the unique Perron branch  at $v_{i_0}$ in  $T^{(2)}$, we have 
$$\rho(M_{v_{i_0}})=\rho(M_{v_{i_0}}^{(2)})= \rho(M_{v_{i_0}}^{(2)} (v_{i_0 +1})). $$
 Hence $B_{v_{i_0}}(v_{i_0 +1})$ is the unique Perron branch  at  $v_{i_0}$ in $T$. Similarly, using $\rho(M_{v_{i_0+1}}^{(1)})>\rho(M_{v_{i_0+1}}^{(2)})$ and the fact that $B_{v_{i_0+1}}(v_{i_0}) = B_{v_{i_0+1}}(v)$ is the unique Perron branch  at  $v_{i_0+1}$ in $T^{(1)}$, we have $B_{v_{i_0+1}}(v_{i_0})$ is the unique Perron branch  at  $v_{i_0+1}$ in $T$. Further, if  $w$ is a vertex other than $v_{i_0}$ and  $v_{i_0+1} $,  it can be seen that the unique Perron branch  at $w$ in $T$ is the branch which contains $ v_{i_0}$ and  $v_{i_0+1}$ by proceeding in a manner  similar to that in  case (a).

The other possible cases are  ($i$) $  \mathcal{C}_{T^{(1)}}=\{v\}$ and   $\mathcal{C}_{T^{(2)}} = \{x\}$, ($ii$)  $  \mathcal{C}_{T^{(1)}}=\{u,v\}$ and   $\mathcal{C}_{T^{(2)}} = \{x,y\}$. It can be seen that the proof follows analogously to the   above cases and hence we omit the details. Combining the conclusion from all the  above cases,  the desired result follows. 
\end{proof}

By combining the results of Lemmas~\ref{lem:s=2-1} and~\ref{lem:s=2-2}, we have shown that  Results~\ref{result:1} and~\ref{result:2}  hold  if the edge weights of a tree  are  $2 \times 2$ lower triangular matrices with positive diagonal entries.  Before proceeding further, we state a few  observations  from the  proof of  Lemmas~\ref{lem:s=2-1} and~\ref{lem:s=2-2} in the following remark.
  
\begin{rem}\label{rem:1-2-lt}
\begin{enumerate}
\item[$1.$] The Results~\ref{result:1} and~\ref{result:2} are valid, if the weights on the edges of the tree $T$ are    $2 \times 2$ lower triangular  matrices  with positive diagonal entries. 

\item[$2.$] The characteristic-like vertex (or vertices) of $T$  lie in the path joining characteristic vertices of $T^{(1)}$ and  $T^{(2)}$.

\item[$3.$] The arguments used to prove  Lemmas~\ref{lem:s=2-1} and~\ref{lem:s=2-2} are  summarized  as follows:
\begin{itemize}
\item[(a)] For any $v\in V$, if  $B$ is a branch of $T$ at $v$,  then $M_v(B)$ is a $2 \times 2$ lower triangular block matrix, {\it i.e.,}
$$M_v(B)=\left[
\begin{array}{c|c}
M_v^{(1)}(B) &  \mathbf{0} \\
\midrule
\ast & M_v^{(2)}(B)
\end{array}
\right].$$
 Hence,  $\rho(M_v(B))=\max\{\rho(M_v^{(1)}(B)), \rho(M_v^{(2)}(B))\}$.

\item[(b)]Results~\ref{result:1} and~\ref{result:2} are true for both  $T^{(1)}$ and  $T^{(2)}$.

\item[(c)]  For  any $u,v,w \in V$, if $B_u(w)\subsetneq B_v(w)$, then $\rho(M_u^{(j)}(w))< \rho(M_v^{(j)}(w))$ \textup{for} $j=1,2.$
\end{itemize}

\end{enumerate}
\end{rem} 

Before proving the results for the general case, we prove a lemma analogous to  Proposition~\ref{prop:big-branch-scalar}.

\begin{lem}\label{lem:big-branch-sr-lt}
Let $T=(V,E)$ be  a tree such that   weights on the edges of $T$ are  lower  triangular matrices   with positive diagonal entries. For  any $u,v,w \in V$, if $B_u(w)\subsetneq B_v(w)$, then $\rho(M_u(w))< \rho(M_v(w)).$
\end{lem} 
\begin{proof}
Let the  weights on the edges of $T$ be $s\times s$  lower  triangular matrices    with positive diagonal entries. For each $1\leq j \leq s$,  let $ T^{(j)}$   be the tree with positive edge weights induced by $T=(V,E)$ with $s\times s$ lower triangular matrix edge weights. For each $1\leq j\leq s$, using Proposition~\ref{prop:big-branch-scalar} for the tree  $T^{(j)}$ if $B_u(w)\subsetneq B_v(w)$  then $\rho(M_u^{(j)}(w))< \rho(M_v^{(j)}(w))$.  Hence,  the  result follows from  Equation~\eqref{eqn:sp-rd-lt}.
\end{proof}
 
 \begin{theorem}\label{thm:lt-s*s}
Let $T=(V,E)$ be a tree such that the weights on the edges of $T$ are lower triangular matrices with positive diagonal entries.  Then, one of the following cases occurs: 
\begin{enumerate}
\item[$1.$] There is a unique vertex $v$ such that there are two or more   Perron branches at $v$ in $T$. Moreover, if  $x$ is a vertex other than $v$, then the unique Perron branch at $x$ in $T$ is the branch which contains the vertex $v$.

\item[$2.$] There is  a unique  pair of vertices $u$ and $v$ with $u \sim v$  such that   the Perron branch  at $u$ in $T$ is  the branch containing $v$, while the Perron branch  at $v$  in $T$  is  the branch containing $u$.  Moreover, the unique Perron branch at any vertex $x$ in $T$ is the branch which contains at least one of the vertices $u$ or $v$.
\end{enumerate}
\end{theorem}
 \begin{proof}
 Let the edges of $T$ be assigned   with   lower triangular  matrix weights of order $s \times s$ with positive diagonal entries. We prove this result using induction on $s$. By Lemmas~\ref{lem:s=2-1} and~\ref{lem:s=2-2}, the result is true for $s=2$. Let us assume that  the result is true whenever matrix weights are of order $(s-1)\times (s-1)$.

Let $\{W(e)\}_{e\in E}$ denote the lower triangular  matrix weights  on $T$ of order $s \times s$ with positive diagonal entries and let $W^*(e)$ denote the principal submatrix  of $W(e)$ corresponding to the indices $1,2,\ldots, s-1$. Let $T^*$ denote the tree $T=(V,E)$ with the  matrix weights $\{W^*(e)\}_{e\in E}$ of order  $(s-1)\times (s-1)$. Then, by the induction hypothesis   Results~\ref{result:1} and~\ref{result:2} hold true for the tree $T^*$ . 
 
Now, we consider  matrix weights of order $s\times s$. For $v\in V$ with $\deg(v)=r$ and for  $1\leq i\leq r$,  let  $B_i$ be the branches  at $v$. Then,  by  Equation~\eqref{eqn:bottle-lt},   we have
$$M_v^{*}(B_i)\simeq \widetilde{M_v^{*}(B_i)}=\left[
\begin{array}{c |c| c| c}
    M_v^{(1)}(B_i) & \mathbf{0}& \dots & \mathbf{0} \\
    \midrule
    \ast & M_v^{(2)}(B_i) & \dots & \mathbf{0} \\
    \midrule    
    \vdots & \vdots & \ddots & \vdots \\
    \midrule
    \ast  & \ast & \dots & M_v^{(s-1)}(B_i)
\end{array}
\right],$$
and
\begin{equation*}
{\small
M_v(B_i)\simeq \widetilde{M_v(B_i)}=\left[
\begin{array}{c |c| c| c}
    M_v^{(1)}(B_i) & \mathbf{0}& \dots & \mathbf{0} \\
    \midrule
    \ast & M_v^{(2)}(B_i) & \dots & \mathbf{0} \\
    \midrule    
    \vdots & \vdots & \ddots & \vdots \\
    \midrule
    \ast  & \ast & \dots & M_v^{(s)}(B_i)
\end{array}
\right]= \left[
\begin{array}{c|c}
\widetilde{M_v^{*}(B_i)} &  \mathbf{0} \\
\midrule
\ast &  M_v^{(s)}(B_i)
\end{array}
\right]. }
\end{equation*}
This implies that  $\rho(M_v(B_i))= \max\{ \rho(M_v^{*}(B_i)), \rho( M_v^{(s)}(B_i))\}$ for all  $1\leq i\leq r$. Therefore, in view of Remark~\ref{rem:1-2-lt}, Lemma~\ref{lem:big-branch-sr-lt} and the induction hypothesis, the desired result follows by proceeding  in exactly the same manner as Lemmas~\ref{lem:s=2-1} and~\ref{lem:s=2-2}. \end{proof}
 \begin{cor}
 Let $T=(V,E)$ be  a tree such that the  weights on the edges of $T$ are    $s\times s$ lower  triangular matrices  with positive diagonal entries. For $1\leq j\leq s$, let $ T^{(j)}$ be the trees with positive edge weights induced by $T=(V,E)$ with  $s\times s$ lower triangular matrix edge weights.  Then, the characteristic-like vertex (or vertices) of $T$ lies in the minimal sub tree of $T$ containing the characteristic vertices of $T^{(j)}$ for $1\leq j\leq s$. 
 \end{cor}
 \begin{proof}
  We use induction on $s$ to prove the result. The result is true for $s=2$ (see Remark~\ref{rem:1-2-lt} $(2)$). Let us assume that  the result is true whenever matrix weights are of order $(s-1)\times (s-1)$. Let $T^*$ be the tree $T$ with weights of order $(s-1)\times (s-1)$ as defined in  proof of Theorem~\ref{thm:lt-s*s}. Then, by the induction hypothesis, characteristic-like vertex (or vertices) of $T^*$ lie in the minimal sub tree of $T$ containing all the characteristic vertices of $T^{(j)}$ for $1\leq j\leq (s-1).$ Thus, proceeding in a manner similar to  that in Lemmas~\ref{lem:s=2-1} and~\ref{lem:s=2-2} for the trees $T^*$ and $T^{(s)}$, we see that the characteristic-like vertex (or vertices) of $T$ lie in the path joining the  characteristic-like vertex (or vertices) of $T^*$ and the characteristic vertex (or vertices) of $T^{(s)}$. Hence, 
  the desired result follows.
\end{proof}
 
From Propositions~\ref{prop:ch-edge} and~\ref{prop:ch-vertex}, we have seen that the first non-zero eigenvalue of the Laplacian matrix  (algebraic connectivity) of a tree with positive weights can be expressed in terms of Perron values. In the next section, we attempt to find a similar relation for trees with matrix weights. However, here we obtain an inequality instead.

\section{Lower Bound on the First Non-zero Laplacian Eigenvalue}\label{sec:lower-bound}

In the literature, the algebraic connectivity plays an important role in understanding the geometry of a tree with positive edge weights (for example, see~\cite{Abreu,Grone2, Kirkland1, Lal, Patra1, Patra2,Wang, Zhang}). In particular,  the understanding of the characteristic vertex (vertices) via Perron value and Perron branch, and the representation of the algebraic connectivity in terms of Perron values yielded several interesting results related to the structure of a tree (for example, see~\cite{Abreu, Kirkland1, Lal, Patra1, Patra2}). Given Section~\ref{sec:char-perron}, we are interested in a similar representation for the first non-zero eigenvalue of  Laplacian matrices via Perron values for trees with matrix edge weights. However, we obtain a  lower bound involving a similar expression on  Perron values.

Let $T$ be a tree on $n$ vertices with either of the following classes of matrix weights on its edges:  $(1)$~positive definite matrix weights,  $(2)$ lower (or upper) triangular matrix weights with positive diagonal entries. From the previous sections, we know that the eigenvalues of the Laplacian matrix $L(T)$ are nonnegative. Moreover, if the matrix weights assigned to the edges of $T$ are of order $s\times s$, then by~\cite[Theorem $2.4$]{Atik}, we have rank$(L(T))=(n-1)s.$  Therefore, if the eigenvalues of $L(T)$ are ordered as in Equation~\eqref{eqn:ev-hermitian}, then $\lambda_{s+1}(L(T))$ is the first non-zero eigenvalue of $L(T)$.  For notational consistency, we denote the first non-zero eigenvalue $\lambda_{s+1}(L(T))$  as $\mu(T)$   (similar to the case of trees with positive edge weights). In this section, we provide a lower bound for $\mu(T)$ in terms of Perron value.  Before proceeding further, using arguments similar to the proof of \cite[Theorem 1]{Kirkland}, we extend the result for trees with the above classes of matrix edge weights.

\begin{lem}\label{lem:ch-edge-avg}
Let $T$ be a tree  with either of the  following  classes of matrix weights on its edges: 
$1.$~positive definite matrix weights,  $2.$ lower (or upper) triangular matrix weights with positive diagonal entries. 
If $T$ has a  characteristic-like edge $e$ between the vertices $u$ and $v$,   
 then $\exists$ $0< \nu < 1$ such that $$\rho(M_u(v) - \nu ( J\otimes [W(e)^{-1}])) = \rho(M_v(u) - (1-\nu) ( J\otimes [W(e)^{-1}])),$$ where $W(e)$ denotes the matrix weight on the edge $e$.
\end{lem}
\begin{proof}
Let $\widehat{M}_u$ denote  the principal submatrix of $M_u$ obtained by deleting the block $M_u(v)$ (the block corresponding to the unique Perron branch $B_u(v)$  at $u$ in $T$) from $M_u$. Similarly, let $\widehat{M}_v$ denote the principal submatrix of $M_v$ obtained by deleting the block $M_v(u)$ (the block corresponding to the unique Perron branch $B_v(u)$  at $v$ in $T$) from $M_v$, {\it i.e.,}  
$$M_u=\left[
\begin{array}{c|c}
M_u(v) &  \mathbf{0} \\
\midrule
\mathbf{0}  & \widehat{M}_u
\end{array}
\right] \mbox{ and } 
M_v=\left[
\begin{array}{c|c}
\widehat{M}_v &  \mathbf{0} \\
\midrule
\mathbf{0}  &  M_v(u)
\end{array}
\right].$$
Then
\begin{equation}\label{eqn:4.1}
\rho(M_u(v)) > \rho(\widehat{M}_u)\mbox{ and } \rho(M_v(u)) > \rho(\widehat{M}_v).
\end{equation}
Further,  
\begin{equation}\label{eqn:M-d-hat}
\widehat{\widehat{M\, }}_u =M_v(u) - J\otimes [W(e)^{-1}] \mbox{ and } \widehat{\widehat{M\, }}_v =M_u(v) - J\otimes [W(e)^{-1}],
\end{equation}
where
$$\widehat{\widehat{M\, }}_u=\left[
\begin{array}{c|c}
\mathbf{0}_{s\times s} &  \mathbf{0} \\
\midrule
\mathbf{0}  & \widehat{M}_u
\end{array}
\right] \mbox{ and } 
\widehat{\widehat{M\, }}_v=\left[
\begin{array}{c|c}
\widehat{M}_v &  \mathbf{0} \\
\midrule
\mathbf{0}  & \mathbf{0}_{s\times s}
\end{array}
\right],$$
when the matrix weights on  edges are of order $s\times s$. Thus $\rho(\widehat{\widehat{M\, }}_u)=\rho(\widehat{M}_u)$ and $\rho(\widehat{\widehat{M\, }}_v)=\rho(\widehat{M}_v)$. Using Equations~\eqref{eqn:4.1} and~\eqref{eqn:M-d-hat}, we have
\begin{equation}\label{eqn:4.2}
\rho(M_u(v)) > \rho(M_v(u) - J\otimes [W(e)^{-1}])\mbox{ and } \rho(M_v(u)) > \rho(M_u(v) - J\otimes [W(e)^{-1}]) .
\end{equation}
For $0\leq t \leq 1$, let
\begin{equation*}\label{eqn:f-g}
\begin{cases}
\vspace*{.1cm}
f(t)= \rho(M_u(v)- t( J\otimes [W(e)^{-1}])),\\
g(t)= \rho(M_v(u)- (1-t)( J\otimes [W(e)^{-1}])).
\end{cases}
\end{equation*}
Then,
\begin{itemize}
\item for positive definite matrix weights, $J\otimes [W(e)^{-1}]$ is a positive semidefinite matrix. Then using the min-max theorem we see that $f(t)$ is a continuous, decreasing function  and $g(t)$  is a continuous, increasing function.

\item for lower triangular matrix weights of order $s\times s$ with positive diagonal entries, let $W(e)=[W_{ij}(e)]$.   By Equation~\eqref{eqn:sp-rd-lt},  we have 
$$f(t)=\rho(M_u(v)- t( J\otimes [W(e)^{-1}]))= \max_{1\leq j\leq s}\rho(M_u^{(j)}(v)- t \left( 1/W_{jj}(e)\right)J).$$
Since $\rho(M_u^{(j)}(v)- t( \frac{1}{W_{jj}(e)})J)$ is a decreasing function with respect to $t$ for  all $0\leq t \leq 1$ and $1\leq j\leq s$, we see that $f(t)$ is a continuous, decreasing function. Similarly, it can be seen that $g(t)$ is a continuous,  increasing function.
\end{itemize}
Note that, in the above cases the continuity of functions $f(t)$ and $g(t)$ for $0\leq t \leq 1$, follows from~\cite[Corollary  VI.1.6]{Bhatia}.  Further, $f(t)$ decreases from $\rho(M_u(v))$ to $\rho(M_u(v)-  J\otimes [W(e)^{-1}])$ and $g(t)$ increases from  $\rho(M_v(u)-  J\otimes [W(e)^{-1}])$ to $\rho(M_v(u))$. By Equation~\eqref{eqn:4.2}, $f(t)$ and  $g(t)$ must intersect, and hence the result follows.
\end{proof}

In view of Result~\ref{result:1} and Lemma~\ref{lem:ch-edge-avg},  we now define a constant in terms of Perron values for trees with a suitable class of matrix edge weights. 

\begin{defn}
Let $T$ be a tree  with either of the  following  classes of matrix weights on its edges: 
$(1)$~positive definite matrix weights,  $(2)$ lower (or upper) triangular matrix weights with positive diagonal entries. 
We define a constant $\kappa(T)$ as follows:
\begin{enumerate}
\item[$(a)$] If $T$ has a characteristic-like vertex  $v$, then $\kappa(T)= \dfrac{1}{\rho(M_v)}$.

\item[$(b)$] If $T$ has a characteristic-like edge $e$  between the vertices $u$ and $v$, then 
$$\kappa(T)=\frac{1}{\rho(M_u(v) - \nu ( J\otimes [W(e)^{-1}])) }= \frac{1}{\rho(M_v(u) - (1-\nu) ( J\otimes [W(e)^{-1}]))},$$
where  $0< \nu< 1$ as  defined in Lemma~\ref{lem:ch-edge-avg} and $W(e)$ denotes the matrix weight on  edge $e$.  
\end{enumerate}
\end{defn}

To obtain a lower bound on $\mu(T)$ for any tree $T$ with positive definite matrix edge weights, we first prove the following lemmas.

\begin{lem}\label{lem:M-inv}
Let $T$ be a tree  with nonsingular  matrix weights  on its edges. If $e$ is an edge  between the vertices $u$ and $v$, then for  $0< \alpha < 1$, we have
$$\bigg[M_u(v) - \alpha ( J\otimes [W(e)^{-1}])\bigg]^{-1}
= M_u(v)^{-1}+ \mathbf{e}_v \mathbf{e}_v^T \otimes \left[\left(\frac{\alpha}{1-\alpha}\right)W(e)\right],$$
where  $\mathbf{e}_{v}$ is the column vector of  conformal order with $1$ at the $v^{th}$ entry and $0$ elsewhere, and $W(e)$ is the weight on the edge $e$. 
\end{lem}
\begin{proof}
Let $L(T)$ be the Laplacian matrix of $T$ and $\widehat{L}(B_u(v))$ be the principal submatrix of $L(T)$ corresponding to the vertices in the branch $B_u(v)$. By Theorem~\ref{thm:non-sing-inv}, we know that $\widehat{L}(B_u(v))= M_u(v)^{-1}$. Let 
$$X=  M_u(v)^{-1} + \mathbf{e}_v \mathbf{e}_v^T \otimes [-W(e)].$$ Then, the row and column block sums of $X$ are zero. Thus, from Remark~\ref{rem:rem1} $(2)$, we have 
\begin{equation}\label{eqn:x-prod}
X (J\otimes [W(e)^{-1}])= \mathbf{0}.
\end{equation}
Further note that, the column block of $M_u(v)$ (by Theorem~\ref{thm:non-sing-inv}) and $J\otimes [W(e)^{-1}]$ corresponding to the vertex $v$ is $\mathds{1} \otimes [W(e)^{-1}]$. Hence, 
\begin{equation}\label{eqn:pr-W}
\begin{cases}\vspace*{.1cm}
(\mathbf{e}_v \mathbf{e}_v^T \otimes W(e))M_u(v)= \mathbf{e}_v  \mathds{1}^T \otimes I,\\
(\mathbf{e}_v \mathbf{e}_v^T \otimes W(e))(J\otimes [W(e)^{-1}])=\mathbf{e}_v  \mathds{1}^T \otimes I.
\end{cases}
\end{equation}
This implies that 
\begin{equation}\label{eqn:xm}
XM_u(v)= I + \bigg(\mathbf{e}_v \mathbf{e}_v^T \otimes [-W(e)]\bigg)M_v(u)= I+ \mathbf{e}_v \mathds{1}^T \otimes [-I].
\end{equation}
Now,
\begin{align*}
&\left(M_u(v)^{-1}+ \mathbf{e}_v \mathbf{e}_v^T \otimes \left[\left(\frac{\alpha}{1-\alpha}\right)W(e)\right]\right)\bigg(M_u(v) - \alpha ( J\otimes [W(e)^{-1}])\bigg)\\
=&\left(X+ \mathbf{e}_v \mathbf{e}_v^T \otimes \left[\left(\frac{1}{1-\alpha}\right)W(e)\right]\right)\bigg(M_u(v) - \alpha ( J\otimes [W(e)^{-1}])\bigg)\\
=& XM_u(v)- \alpha \; X (J\otimes [W(e)^{-1}])+ \left(\mathbf{e}_v \mathbf{e}_v^T \otimes \left[\left(\frac{1}{1-\alpha}\right)W(e)\right]\right)M_u(v) \\
&\hspace*{6cm} 
-\alpha \left(\mathbf{e}_v \mathbf{e}_v^T \otimes \left[\left(\frac{1}{1-\alpha}\right)W(e)\right]\right)(J\otimes [W(e)^{-1}]).
\end{align*}
Using Equations~\eqref{eqn:x-prod}, \eqref{eqn:pr-W} and~\eqref{eqn:xm}, the above equation reduces to
\begin{align*}
&\left(M_u(v)^{-1}+ \mathbf{e}_v \mathbf{e}_v^T \otimes \left[\left(\frac{\alpha}{1-\alpha}\right)W(e)\right]\right)\bigg(M_u(v) - \alpha ( J\otimes [W(e)^{-1}])\bigg)\\
=& I+ \left(-1+\frac{1}{1-\alpha}-\frac{\alpha}{1-\alpha} \right) \mathbf{e}_v  \mathds{1}^T \otimes I =I.
\end{align*}
Hence, the desired result follows.
\end{proof}

\begin{lem}\label{lem:L-M}
Let $T$ be a tree on $n$ vertices with nonsingular  matrix weights  on its edges and   $L(T)$ be the Laplacian matrix of $T$. If $e$ is an edge  between the vertices $u$ and $v$, then for $0< \alpha < 1$, we have
\begin{equation}\label{eqn:L-M}
 L(T) + \mathbf{E} \otimes W(e)={\small\left[
\begin{array}{c|c}
\bigg[M_u(v) - \alpha ( J\otimes [W(e)^{-1}])\bigg]^{-1} &  \mathbf{0} \\
\midrule
\mathbf{0}  & \bigg[M_v(u) - (1-\alpha) ( J\otimes [W(e)^{-1}])\bigg]^{-1}
\end{array}
\right]},
\end{equation}
where $W(e)$ is the weight on edge $e$ and $\mathbf{E}=[\mathbf{E}_{xy}]_{x,y \in V}$ is an $n \times n$ matrix with
$$\mathbf{E}_{xy}=
\begin{cases}
\vspace*{.15cm}
\dfrac{\alpha}{1-\alpha} & \mbox{ if } x=y=v,\\\vspace*{.05cm}
 \dfrac{1-\alpha}{\alpha}  & \mbox{ if } x=y=u,\\\vspace*{.05cm}
1 & \mbox{ if } x=u,y=v \mbox{ or } x=v,y=u,\\
0 & \mbox{ otherwise. }
\end{cases}
$$
\end{lem}
\begin{proof}
Let  $B_u(v)$ be the branch consisting of $k$ vertices and $B_v(u)$ be the branch  consisting of $(n-k)$ vertices. By suitable rearrangement of the vertex ordering and  from Theorem~\ref{thm:non-sing-inv}, the Laplacian matrix $L(T)$ of $T$ can be written as 
\begin{equation}\label{eqn:L-part-uv}
L(T)= \left[
\begin{array}{c|c}
M_u(v)^{-1} &  E_{k1} \otimes [-W(e)] \\
\midrule
 E_{1k} \otimes [-W(e)]   & M_v(u)^{-1}
\end{array}
\right],
\end{equation}
where $E_{k1}$ is the $k \times (n-k)$ matrix with $1$ at the $(k,1)^{th}$ position and $0$ elsewhere, and $E_{1k}$ is its transpose. Observe that,  the partitioning here is such that the last row of $M_u(v)^{-1}$ corresponds to the vertex $v$, whereas the first row of $ M_v(u)^{-1}$ corresponds to the vertex $u$.

For $0<\alpha<1$, using Lemma~\ref{lem:M-inv}, we have 
\begin{equation*}
\begin{cases}\vspace*{.2cm}
\ds M_u(v)^{-1}=\bigg[M_u(v) - \alpha ( J\otimes [W(e)^{-1}])\bigg]^{-1}- \  \mathbf{e}_v \mathbf{e}_v^T \otimes \left[\left(\frac{\alpha}{1-\alpha}\right)W(e)\right],\\
\ds M_v(u)^{-1}=\bigg[M_v(u) - (1-\alpha) ( J\otimes [W(e)^{-1}])\bigg]^{-1}- \  \mathbf{e}_v \mathbf{e}_v^T \otimes \left[\left(\frac{1-\alpha}{\alpha}\right)W(e)\right].
\end{cases}
\end{equation*}
Substituting the above values in Equation~\eqref{eqn:L-part-uv}, the desired result follows.
\end{proof}
\begin{rem}
\begin{enumerate}
\item From the proofs of  Lemmas~\ref{lem:M-inv} and \ref{lem:L-M}, it is easy to see that the result also applies for any real $\alpha$, where $\alpha \neq 1$ and $\alpha\neq 0.$
\item In Lemma~\ref{lem:L-M}, the matrix $\mathbf{E}$ is a rank one matrix, and for $0<\alpha <1$, its  only non-zero eigenvalue  is positive.
\end{enumerate}
\end{rem}

Now we prove the result which gives a lower bound on $\mu(T)$   whenever the edges of the tree $T$ are assigned with positive definite matrix weights.

\begin{theorem}\label{thm:lower-bound-pd}
Let $T$ be  a tree  with  positive definite matrix  weights  on its edges. Let $L(T)$ be the Laplacian matrix of $T$ and  $\mu(T)$ be the  first non-zero eigenvalue of  $L(T)$. Then $\kappa(T) \leq \mu(T)$.
\end{theorem}
\begin{proof}
Let $T$ be a tree on $n$ vertices and  the  weights on the edges of $T$ be   $s\times s$  positive definite matrices. Then $L(T)$ is a symmetric matrix of order $ns\times ns$. Let the eigenvalues of $L(T)$ be ordered as in Equation~\eqref{eqn:ev-hermitian}. Thus $\mu(T)= \lambda_{s+1}(L(T)).$

 If $T$ has a characteristic-like vertex, say $v$, then $\kappa(T)=1/\rho(M_v) $. Let $L_{v}$ be the principal submatrix of $L(T)$ obtained by deleting the row  block and column block corresponding to  the  vertex $v$. Let the eigenvalues of $L_{v}$ be ordered as in Equation~\eqref{eqn:ev-hermitian}. Since $L_v= M_v^{-1}$, we see that $\lambda_1(L_{v})=\kappa(T).$ Using Theorem~\ref{thm:inclu}, for the principal submatrix $L_v$ of order $(n-1)s\times (n-1)s$, we have
 $$0=\lambda_1(L(T))\leq  \lambda_1(L_{v}) \leq \lambda_{1+ns-(n-1)s}(L(T))=\lambda_{s+1}(L(T)). $$
Hence, $\kappa(T) \leq \mu(T)$.
 
  If $T$ has a characteristic-like edge $e$  between the vertices $u$ and $v$, then  $\exists$~$0< \nu< 1$ such that
  \begin{equation}\label{eqn:k}
  \kappa(T)=\frac{1}{\rho(M_u(v) - \nu ( J\otimes [W(e)^{-1}])) }= \frac{1}{\rho(M_v(u) - (1-\nu) ( J\otimes [W(e)^{-1}]))}.
  \end{equation}
By Lemma~\ref{lem:L-M}, Equation~\eqref{eqn:L-M} holds true for $\alpha=\nu.$  Since the edges of $T$ are assigned with positive definite matrices, both $L(T)$ and $\mathbf{E} \otimes W(e)$ are real symmetric matrices. Using Theorem~\ref{thm:weyl}, we have 
\begin{equation}\label{eqqn:ineq-L+E}
\lambda_1(L(T) + \mathbf{E} \otimes W(e))\leq \lambda_{s+1}(L(T))+\lambda_{(n-1)s}(\mathbf{E}\otimes W(e)).
\end{equation}
For $\alpha=\nu$, from Equations~\eqref{eqn:L-M} and~\eqref{eqn:k}, we have $\lambda_1(L(T) + \mathbf{E} \otimes W(e))= \kappa(T).$ Further, note that $\mathbf{E}$ is a rank one matrix. Since $0< \nu< 1$, from  Remark~\ref{rem:rem1} we see that $\mathbf{E}\otimes W(e)$ is a positive semidefinite matrix with rank$(\mathbf{E} \otimes W(e))=$ rank$(W(e))=s.$ This implies that $\lambda_i(\mathbf{E} \otimes W(e))=0$ for all $1\leq i \leq (n-1)s.$ Hence, Equation~\eqref{eqqn:ineq-L+E} reduces to  $\kappa(T) \leq \mu(T)$ and this completes the proof. 
\end{proof}

We now prove the result that gives a lower bound on $\mu(T)$  whenever the edges of the tree $T$ are assigned lower triangular matrix weights with positive diagonal entries.

\begin{theorem}\label{thm:lower-bound-lt}
Let $T=(V,E)$ be  a tree such that the  weights on the edges of $T$ are    lower  triangular matrices   with positive diagonal entries. Let $L(T)$ be the Laplacian matrix of $T$ and  $\mu(T)$ be the  first non-zero eigenvalue of  $L(T)$. Then $\kappa(T) \leq \mu(T)$.
\end{theorem} 
\begin{proof}
Let the  weights on the edges of $T$ be $s\times s$  lower  triangular matrices   with positive diagonal entries. For $1\leq j\leq s$, let $ T^{(j)}$ be the trees with positive edge weights induced by $T=(V,E)$ with  $s\times s$ lower triangular matrix edge weights. For $1\leq j\leq s$, let  $L(T^{(j)})$ be the Laplacian matrix of  $T^{(j)}$. Then,  using Equation~\eqref{eqn:lap-ltw},  we have $$\sigma(L(T))= \bigcup_{j=1}^s \sigma(L(T^{(j)})) \mbox{ and }
\mu(T)= \min_{1\leq j\leq s} \mu(T^{(j)}),$$
where $\mu(T^{(j)})$ denotes  the algebraic connectivity of $T^{(j)}$. 

Without loss of generality, let us assume $\mu(T)=\mu(T^{(1)})$. We now consider the following cases to complete the proof.\\

\noindent$\underline{\textbf{Case 1:}}$ Let $T$ have a characteristic-like vertex, say $v$. Then, $\kappa(T)=\dfrac{1}{\rho(M_v)}$.\\

$\underline{\textbf{Subcase 1.1:}}$ Let $T^{(1)}$ have a characteristic vertex, say $x$. By Proposition~\ref{prop:ch-vertex}, there are two or more Perron branches   at $x$ in  $T^{(1)}$ and hence there exists a vertex $y$ adjacent to $x$ (and $y$ is not in the path  $\mathcal{P}(v,x)$ if $v\neq x$) such that $B_x(y)$ is a Perron branch of at $x$ in  $T^{(1)}$. Thus,
$$B_x(y)\subseteq B_v(y) \mbox{ and }\rho(M_x^{(1)})=\rho(M_x^{(1)}(y)).$$
Using Proposition~\ref{prop:ch-vertex}, Lemma~\ref{lem:big-branch-sr-lt} and Equation~\eqref{eqn:perr}, we have
$$\frac{1}{\kappa(T)}= \rho(M_v) \geq \rho(M_v(y))
                               \geq \rho(M_x(y))
                               \geq \rho(M_x^{(1)}(y))
                               =\rho(M_x^{(1)})
                               =\frac{1}{\mu(T^{(1)})}
                                =\frac{1}{\mu(T)}.$$
                                
$\underline{\textbf{Subcase 1.2:}}$ Let $T^{(1)}$ have a characteristic edge $\hat{e}$ between the vertices $x$ and $y$. Using Proposition~\ref{prop:ch-edge}, there exists  $0< \gamma < 1$ such that
\begin{equation}\label{eqn:ch-edge-ineq}
\frac{1}{\mu(T^{(1)})}=\rho(M_x^{(1)}(y) - \gamma (1/\theta) J) = \rho(M_y^{(1)}(x) - (1-\gamma)  (1/\theta) J),
\end{equation}
where $\theta$ is the positive  weight assigned to the edge $\hat{e}$ in $T^{(1)}$. Without loss of generality,  let $y$ not be in the path  $\mathcal{P}(v,x)$. Here $B_x(y)\subseteq B_v(y)$. Using Lemma~\ref{lem:big-branch-sr-lt}, Equations~\eqref{eqn:perr} and~\eqref{eqn:ch-edge-ineq}, we have
$$\frac{1}{\kappa(T)}= \rho(M_v) \geq \rho(M_v(y))
                               \geq \rho(M_x(y))
                               \geq \rho(M_x^{(1)}(y))
           >\rho(M_x^{(1)}(y) - \gamma (1/\theta) J)
                               =\frac{1}{\mu(T^{(1)})}
                                =\frac{1}{\mu(T)}.$$

 \noindent$\underline{\textbf{Case 2:}}$  Let $T$ have a  characteristic-like edge $e$  between the vertices $u$ and $v$. For $0\leq t \leq 1$, let
\begin{equation}\label{eqn:f-g-h}
\begin{cases}
\vspace*{.1cm}
f(t)= \rho(M_u(v)- t( J\otimes [W(e)^{-1}])),\\\vspace*{.1cm}
g(t)= \rho(M_v(u)- (1-t)( J\otimes [W(e)^{-1}])),\\
h(t)=\min\{f(t),g(t)\}.
\end{cases}
\end{equation}
From the proof of Lemma~\ref{lem:ch-edge-avg}, we know that $f(t)$ is a continuous, decreasing function and  $g(t)$ is a continuous,  increasing function. Hence, there exists  $0< \nu < 1$ such that $f(\nu)=g(\nu)$, {\it i.e.,} 
$$\rho(M_u(v) - \nu ( J\otimes [W(e)^{-1}])) = \rho(M_v(u) - (1-\nu) ( J\otimes [W(e)^{-1}])),$$
where $W(e)$ is the matrix weight on  the edge $e$. Therefore, 
\begin{equation}\label{eqn:h}
 h(\nu)= \max_{0\leq t\leq 1} h(t)=\max_{0\leq t\leq 1} \min\{f(t),g(t)\}=f(\nu)=g(\nu).
\end{equation} 
By definition,
  $$\kappa(T)=\frac{1}{\rho(M_u(v) - \nu ( J\otimes [W(e)^{-1}])) }= \frac{1}{\rho(M_v(u) - (1-\nu) ( J\otimes [W(e)^{-1}]))},$$
and hence by Equation~\eqref{eqn:h}, we have 
\begin{equation}\label{eqn:h-k}
 h(\nu)= \max_{0\leq t\leq 1} h(t)=\max_{0\leq t\leq 1} \min\{f(t),g(t)\}=\frac{1}{\kappa(T)}.
\end{equation}
Let $\widehat{M}_v$ denote the  principal submatrix of $M_v$ obtained by deleting the block $M_v(u)$ (the block corresponding to the unique Perron branch $B_v(u)$  at $v$ in $T$) from $M_v$. Then, from Equation~\eqref{eqn:M-d-hat} we see that  $\widehat{\widehat{M\, }}_v= M_u(v) -  J\otimes [W(e)^{-1}]$ and $\rho(\widehat{\widehat{M\, }}_v)=\rho(\widehat{M}_v)$. Hence,
\begin{equation}\label{eqn:ch-edge-1}
\frac{1}{\kappa(T)}= \rho(M_u(v) - \nu ( J\otimes [W(e)^{-1}])) > \rho(M_u(v) -  J\otimes [W(e)^{-1}]) =  \rho(\widehat{M}_v).   
\end{equation}   

$\underline{\textbf{Subcase 2.1:}}$ Let $T^{(1)}$ have a characteristic vertex, say $x$. Without loss of generality, let us assume $x\in B_u(v)$. By Proposition~\ref{prop:ch-vertex}, there are two or more Perron branches   at $x$ in  $T^{(1)}$.  Hence, there exists a vertex $y$ adjacent to $x$ (and $y$ is not in the path  $\mathcal{P}(v,x)$ if $v\neq x$) such that $B_x(y)$ is a Perron branch  at $x$ in  $T^{(1)}$. Thus,
$$B_x(y)\subseteq B_v(y) \mbox{ and }\rho(M_x^{(1)})=\rho(M_x^{(1)}(y)).$$
Note that, $\widehat{M}_v$ is a block diagonal matrix and $M_v(y)$ is one of its blocks. Thus, $\rho(\widehat{M}_v) \geq \rho(M_v(y))$. Hence,  using  Lemma~\ref{lem:big-branch-sr-lt}, Equations~\eqref{eqn:perr} and~\eqref{eqn:ch-edge-1}, we have               
$$\frac{1}{\kappa(T)}> \rho(\widehat{M}_v) \geq \rho(M_v(y)) \geq \rho(M_x(y))\geq \rho(M_x^{(1)}(y))=\rho(M_x^{(1)})
                               =\frac{1}{\mu(T^{(1)})}
                                =\frac{1}{\mu(T)}.$$
                               
$\underline{\textbf{Subcase 2.2:}}$ Let $T^{(1)}$ have a characteristic edge $\hat{e}$ between the vertices $x$ and $y$.   Thus, Equation~\eqref{eqn:ch-edge-ineq} is valid. 

Let $ e\neq \hat{e}$.  Without loss of generality, let $x,y\in B_u(v)$  and $y$  not be in the path  $\mathcal{P}(v,x)$ if $v\neq x$. Here, $B_x(y)\subseteq B_v(y)$.                       
 Using  Lemma~\ref{lem:big-branch-sr-lt}, Equations~\eqref{eqn:perr},~\eqref{eqn:ch-edge-ineq}  and~\eqref{eqn:ch-edge-1}, we have   
 $$\frac{1}{\kappa(T)}> \rho(\widehat{M}_v) \geq \rho(M_v(y))
                               \geq \rho(M_x(y))
                               \geq \rho(M_x^{(1)}(y))
           >\rho(M_x^{(1)}(y) - \gamma  (1/\theta) J)
                               =\frac{1}{\mu(T^{(1)})}
                                =\frac{1}{\mu(T)}.$$   
                                                            
Let $e = \hat{e}$.  Without loss of generality, let us assume that $u=x$  and $v=y$. Thus, Equation~\eqref{eqn:ch-edge-ineq} can be rewritten as
\begin{equation}\label{eqn:ch-edge-ineq11}
\frac{1}{\mu(T^{(1)})}=\rho(M_u^{(1)}(v) - \gamma (1/\theta) J) = \rho(M_v^{(1)}(u) - (1-\gamma)  (1/\theta) J) \mbox{ for some }0< \gamma < 1.
\end{equation}     
Using Equations~\eqref{eqn:f-g-h} and~\eqref{eqn:ch-edge-ineq11}, we have
\begin{align*}
f(\gamma)= \rho(M_u(v)- \gamma( J\otimes [W(e)^{-1}]))\geq \rho(M_u^{(1)}(v) - \gamma (1/\theta) J)=\frac{1}{\mu(T^{(1)})}.
\end{align*}
Similarly, $\ds g(\gamma) \geq \frac{1}{\mu(T^{(1)})}$. Therefore,
$\ds h(\gamma)=\min\{f(\gamma),g(\gamma)\} \geq \frac{1}{\mu(T^{(1)})}$.  Using Equation~\eqref{eqn:h-k}, we have
$$\frac{1}{\kappa(T)}=h(\nu)= \max_{0\leq t\leq 1} h(t) \geq  h(\gamma) \geq \frac{1}{\mu(T^{(1)})}=\frac{1}{\mu(T)}.$$
This completes the proof.
\end{proof}

\begin{figure}[ht]
 \begin{subfigure}{0.45\textwidth}

\centering
\begin{tikzpicture} [scale=1]
\Vertex[x=-2,y=1,label=$v_1$]{A}
\Vertex[x=-2,y=-1,label=$v_2$]{B}
\Vertex[x=0,y=0,label=$v_3$]{C}
\Vertex[x=2,y=0,label=$v_4$]{D}\Vertex[x=4,y=1,label=$v_5$]{E}
\Vertex[x=4,y=-1,label=$v_6$]{F}

\Edge[label=$e_1$,position=above](A)(C)
\Edge[label=$e_2$,position=above](B)(C)
\Edge[label=$e_3$,position=above](C)(D)
\Edge[label=$e_4$,position=above](E)(D)
\Edge[label=$e_5$,position=above](F)(D)
\end{tikzpicture} 

    \caption{}
    \label{fig:1}
  \end{subfigure} \hspace{7mm}
   \begin{subfigure}{0.4\textwidth}
   \vspace*{.7cm}  
   \centering
\begin{tikzpicture}[scale=0.9]
\Vertex[x=-4,y=0,label=$v_1$]{A}
\Vertex[x=-2,y=0,label=$v_2$]{B}
\Vertex[x=0,y=0,label=$v_3$]{C}
\Vertex[x=2,y=0,label=$v_4$]{D}\Vertex[x=4,y=0,label=$v_5$]{E}

\Edge[label=$e_1$,position=above](A)(B)
\Edge[label=$e_2$,position=above](B)(C)
\Edge[label=$e_3$,position=above](C)(D)
\Edge[label=$e_4$,position=above](E)(D)
\end{tikzpicture} 
   \vspace*{.5cm}
    \caption{}
    \label{fig:2}
  \end{subfigure}
  \caption{}\label{fig:3}
\end{figure}

From~\cite{Kirkland}, we know that the equality is attained in Theorems~\ref{thm:lower-bound-pd} and~\ref{thm:lower-bound-lt} for trees with positive edge weights,  but in general, this may not be true. This is illustrated in the following examples.

\begin{ex}
Let $V=\{v_1,v_2,v_3,v_4,v_5, v_6\}$ and  $E=\{e_1,e_2,e_3,e_4,e_5\}$. Consider the tree $T=(V,E)$, as shown in  Figure~\eqref{fig:1} with the matrix weights 
$$\mathcal{W}=\left\{ W(e_1)=W(e_2)=\left[
\begin{array}{c c}
1 &  0 \\
0  & 10
\end{array}
\right], W(e_3)=\left[
\begin{array}{c c}
10 &  0 \\
0  & 10
\end{array}
\right], W(e_4)=W(e_5)=\left[
\begin{array}{c c}
10 &  0 \\
0  & 1
\end{array}
\right]\right\}.$$
Let 
$$\mathcal{W}^{(1)}=\left\{ W(e_1)=W(e_2)=1, W(e_3)= 10, W(e_4)=W(e_5)=10\right\},$$  $$\mathcal{W}^{(2)}=\left\{ W(e_1)=W(e_2)=10, W(e_3)= 10, W(e_4)=W(e_5)=1\right\}.$$ 
Let    $T^{(1)}=(T, \mathcal{W}^{(1)} )$ and $T^{(2)}=(T, \mathcal{W}^{(2)})$.  Then,  
$v_3$ is the characteristic vertex of $T^{(1)}$ with  $\mu(T^{(1)})=1$, while $v_4$ is the characteristic vertex of $T^{(2)}$ with $\mu(T^{(2)})=1$. Whereas, $e_3$ is the characteristic-like edge of $T$ with $\mu(T)=1$. Thus,
$$\frac{1}{\kappa(T)}=\rho(M_{v_3}(v_4)-0.5 (J\otimes[W(e_3)^{-1}] ))=\rho(M_{v_4}(v_3)-0.5 (J\otimes[W(e_3)^{-1}] ))=1.104741,$$
and  hence $\kappa(T)< \mu(T).$  
\end{ex}

\begin{ex}
Let $V=\{v_1,v_2,v_3,v_4,v_5\}$ and  $E=\{e_1,e_2,e_3,e_4\}$. Consider the tree $T=(V,E)$, as shown in  Figure~\eqref{fig:2} with the matrix weights 
$$\mathcal{W}=\left\{ W(e_1)=W(e_2)=\left[
\begin{array}{c c}
10 &  0 \\
0  & 1
\end{array}
\right], W(e_3)= W(e_4)=\left[
\begin{array}{c c}
1 &  0 \\
0  & 10
\end{array}
\right]\right\}.$$
Let 
$$\mathcal{W}^{(1)}=\left\{ W(e_1)=W(e_2)=10, W(e_3)= W(e_4)=1\right\},$$  $$\mathcal{W}^{(2)}=\left\{ W(e_1)=W(e_2)=1, W(e_3)= W(e_4)=10\right\}.$$ 
Let    $T^{(1)}=(T, \mathcal{W}^{(1)} )$ and $T^{(2)}=(T, \mathcal{W}^{(2)})$. Then,
$e_3$ is the characteristic edge of $T^{(1)}$ with  $\mu(T^{(1)})=0.58963$, while $e_2$ is the characteristic edge of $T^{(2)}$ with $\mu(T^{(2)})=0.58963$. Whereas, $v_3$ is the characteristic-like vertex of $T$ with $\mu(T)=0.58963$. Thus,
$\dfrac{1}{\kappa(T)}=\rho(M_{v_3})=2.618034,$
and  hence $\kappa(T)< \mu(T).$ 
\end{ex}

\section{Conclusion}\label{sec:conclu}
In this manuscript, we have studied the Laplacian matrices for trees with matrix weights on their edges. We consider the principal submatrix $L_v$ of the Laplacian matrices for trees with matrix weights on their edges. We first computed the determinant of $L_v$ and proved that $L_v$ is an invertible matrix if and only if the edge weights are nonsingular matrices. Then, we found the inverse of $L_v$ and defined the bottleneck matrix for a branch of a tree with nonsingular matrix edge weights. In this case, we defined Perron values and Perron branches whenever the eigenvalues of $L_v$  are nonnegative. Using  $L_v^{-1}$, we obtained the Moore-Penrose inverse of the Laplacian matrix $L$.  We then considered trees with the following classes of matrix edge weights:
\begin{enumerate}
\item[$1.$] positive definite matrix weights, 

\item[$2.$] lower (or upper) triangular matrix weights with positive diagonal entries.
\end{enumerate}
For trees with the above classes of matrix edge weights, we found that the eigenvalues of $L_v$  are nonnegative, and we have shown the existence of vertices satisfying properties analogous to the properties of characteristic
vertices of trees with positive edge weights in terms of Perron values and Perron branches.  We call such vertices characteristic-like vertices. 

For trees with positive edge weights, it is known that the algebraic connectivity (first non-zero eigenvalue of the Laplacian matrix)  can be expressed in terms of the Perron value. We attempted to find a similar relation for trees with the above class of matrix edge weights. However, here we obtained an inequality instead and hence provided a lower bound for the first non-zero eigenvalue of the Laplacian matrix.

\vspace*{.5cm}

\noindent{ \textbf{\Large Acknowledgements}}: 
We take this opportunity to  thank the anonymous referees for their critical reading of the manuscript and for various suggestions which greatly improve the presentation
of this paper. We also like to thank Prof. R.B. Bapat for his helpful comments and suggestions.  We sincerely thank N. Nilakantan and D. Sheetal for a careful reading of the manuscript. A few of the results of this manuscript are part of the first author's BS-MS major project at IISER Thiruvananthapuram. We acknowledge IISER Thiruvananthapuram for the support provided. The second author acknowledges the Department of Science and Technology, Government of India, for support through the project  MATRICS (MTR/2017/000458).

\end{document}